\documentclass[11pt]{article}
\usepackage[utf8]{inputenc}
\usepackage{notes-style}
\usepackage[backend=biber]{biblatex}
\bibliography{./ref}

\title{Framed cohomological Hall algebras and cohomological stable envelopes}
\author{\Large Tommaso Maria Botta$^1$}
\date{%
	$^1$Departement Mathematik, ETH Zürich,
	8092 Zürich, Switzerland\\%
	E-mail: tommaso.botta@math.ethz.ch%
}

\begin{document}
\maketitle


\begin{abstract}
    There are multiple conjectures relating the cohomological Hall algebras (CoHAs) of certain substacks of the moduli stack of representations of a quiver $Q$ to the Yangian $Y^{Q}_{MO}$ by Maulik-Okounkov, whose construction is based on the notion of stable envelopes of Nakajima varieties. In this article, we introduce the cohomological Hall algebra of the moduli stack of framed representations of a quiver $Q$ (framed CoHA) and we show that the equivariant cohomology of the disjoint union of the Nakajima varieties $\naka_Q(\vi,\w)$ for all dimension vectors $\vi$ and framing vectors $\w$ has a canonical structure of subalgebra of the framed CoHA. Restricted to this subalgebra, the algebra multiplication is identified with the stable envelope map. As a corollary, we deduce an explicit inductive formula to compute stable envelopes in terms of tautological classes. \\[2ex] 
\end{abstract}


\tableofcontents


\section{Introduction}

\subsection{Overview}

The purpose of this article is to establish a connection between the theory of stable envelopes and the world of cohomological Hall algebras. 

The theory of stable envelopes, developed by Okounkov and its coauthors in the last decade \cite{maulik2012quantum, okounkov2017enumerative, aganagic2016elliptic, okounkov2020inductiveI, okounkov2020inductiveII}, is the geometric answer to the problem of constructing a quantum group associated with an arbitrary quiver and of studying the monodromy of the associated quantum differential and difference equations. 
In a nutshell, the stable envelopes of a sufficiently nice variety $X$ equipped with a torus action are certain axiomatically-defined correspondences in cohomology that encode the
equivariant geometry of $X$. Stable envelopes can be defined for many different varieties $X$. Still, the case of Nakajima varieties is particularly relevant. Indeed, not only are the latter one of the richest sources of symplectic resolutions, for which stable envelopes enjoy nicer properties, but they also bridge the gap between the abstract theory of stable envelopes and the Lie-theoretic realm of quantum groups. This follows from the fact that a choice of Nakajima variety is, first and foremost, a choice of a quiver, and its associated stable envelopes give rise to R-matrices for the corresponding quantum group.

The second fundamental ingredient for this article is the theory of cohomological Hall algebras (CoHAs). In the literature, many different constructions go under the name of CoHAs. Their common feature is being certain associative algebras whose underlying vector space encodes irreducible representations of a quiver and whose multiplication encodes extensions of representations. The adjective ``cohomological'' distinguishes them from the previous notion of Hall algebras \cite{Ringel1990}, highlighting that their construction depends on the choice of a generalized cohomology theory. Schiffmann and Vasserot introduced the first example of CoHA for the Jordan quiver \cite{schiffmann2012cherednik}. Since then, the theory of CoHAs has rapidly developed and proved influential in many areas of mathematics. There are two main approaches to COHAs. The first one, initiated by Kontsevich and Soibelman \cite{kontsevich2011cohomological} and often referred to as $3d$ CoHAs, is deeply connected with Donaldson-Thomas theory and BPS invariants. The second one, known as $2d$ CoHAs, is related to the first via dimensional reduction \cite{Davison_critical_CoHA} but can be more directly related to quantum groups \cite{yang2018CoHA, Schiffmann2017generators, Schiffmann2017yangians}.

In this article, we introduce a variant of Kontsevich-Soibelman's CoHA that involves moduli stacks of \emph{framed} representations, and we argue that this algebra governs the stable envelopes of Nakajima varieties. More precisely, we realize the cohomology of the disjoint union of all possible Nakajima varieties associated with a quiver $Q$ as a subalgebra of the framed CoHA, and we show that on this subalgebra CoHA's multiplication reduces to stable envelopes. As an application, we exploit the shuffle description of the CoHA product to derive an inductive formula producing tautological presentations of the stable envelopes.

The general philosophy of the geometric representation theory of CoHAs is that taking framings corresponds to studying modules over the (unframed) CoHA. As a consequence, one can interpret this article's results as evidence of the fact that the modules themselves, considered altogether, admit an algebra structure, which we call framed CoHA, and that this algebra contains the algebra of stable envelopes of Nakajima varieties.

\subsubsection{Framed cohomological Hall algebras}

The acronym CoHA hints at a large class of algebras constructed from quiver representations and a given cohomology theory, possibly with some extra data such as some potential. The conceptual starting point of this article is one of the most elementary instances of CoHA, namely Kontsevich-Soibelman's CoHA associated with a doubled quiver $(\overline Q, I)$ without potential. Here, $I$ is the set of vertices of the quiver $Q$. As a vector space, the CoHA is simply defined as 
\begin{equation}
	\label{unframed CoHA}
	\CoHA^{\overline Q}=\bigoplus_{\vi\in \N^I} \CoHA^{\overline Q}(\vi)\qquad\qquad  \CoHA^{\overline Q}(\vi)=H([T^*\Rep(\vi)/G_\vi]),
\end{equation}
where $H([T^*\Rep(\vi)/G_{\vi}])$ is the singular cohomology of  $[T^*\Rep(\vi)/G_{\vi}]$, the moduli stack of representations of the path algebra $\Ci \overline Q$ with dimension vector $\vi\in \N^I$. The multiplication $\mathsf{m}=\mathsf{q}_*\circ \mathsf{p}^*$ is obtained using the stack $Z$ parametrizing pairs of representations $x\subset y$ of dimension $\vi_1$ and $\vi_1+\vi_2$ and the natural maps 
\[
\begin{tikzcd}
	\left[T^*\Rep(\vi_1)/G_{\vi_1}\right]\times \left[T^*\Rep(\vi_2)/G_{\vi_2}\right] & Z\arrow[l, swap, "\mathsf{p}"]\arrow[r, "\mathsf{q}"] & \left[T^*\Rep(\vi_1+\vi_2)/G_{\vi_1+\vi_2}\right].
\end{tikzcd}
\]
With this construction in mind, we introduce a framed CoHA
\[
\CoHA^{\overline Q_{\text{fr}}}=\bigoplus_{\vi,\w\in \N^I}\CoHA^{\overline Q_{\text{fr}}}(\vi,\w)\qquad \qquad  \CoHA^{\overline Q_{\text{fr}}}(\vi,\w)=H_{T_\w}([T^*\Rep(\vi,\w)/G_\vi])
\]
by replacing representations $T^*\Rep(\vi)$ with framed representations $T^*\Rep(\vi,\w)$ and ordinary cohomology with equivariant cohomology with respect to the residual action of the torus $T_\w$ acting on the framing and rescaling the cotangent directions by some additional parameter $\hbar$. The multiplication is defined with a correspondence analogous to the previous one, and hence consists of a collection of maps 
\begin{equation}
	\label{CoHA multiplication maps intro}
	\CoHAQ(\vi_1,\w_1)\otimes_\Bbbk \CoHAQ(\vi_2,\w_2)\to \CoHAQ(\vi_1+\vi_2,\w_1+\w_2)
\end{equation}
In this way, the unframed CoHA $\CoHA^{\overline Q}$ can be naturally identified with the specialization for ${\hbar}=0$ of the subalgebra of $\CoHA^{\overline Q_{\text{fr}}}$ with trivial framing $\w=0$.

\subsubsection{Stable envelopes of Nakajima varieties}

An important class of symplectic varieties associated with framed representations of a quiver is given by Nakajima quiver varieties \cite{Nakajimaquiver, Nakajimainstantons, Nakajima_2001, ginzburg2009lectures}. These are defined as the GIT-symplectic reduction of the cotangent bundle $T^*\Rep(\vi,\w)$ with respect to the Hamiltonian action of $G_\vi$. This means that if we denote the moment map by
\[
\mu_{\vi,\w}: T^*\Rep(\vi,\w)\to \g^*_\vi,
\]
then we can describe a Nakajima variety as the quotient 
\[
\naka(\vi,\w)=\mu_{\vi,\w}^{-1}(0)^{ss}/G_{\vi}.
\]
By construction, $\naka(\vi,\w)$ comes with a torus $T_\w$ acting on the framing and rescaling the symplectic form with weight ${\hbar^{-1}}$.
Stable envelopes of Nakajima quiver varieties are certain axiomatically defined maps 
\[
\text{Stab}_{\mathfrak{C}}: H_{T_\w}(\naka(\vi,\w)^A)\to H_{T_\w}(\naka(\vi,\w))
\]
depending on the choice of a subtorus $A\subset T_\w$ fixing the symplectic form of $\naka(\vi,\w)$ and on  the extra piece of data $\mathfrak{C}$, called chamber. Remarkably, each fixed component of $\naka(\vi,\w)^A$ is itself a product of Nakajima varieties and, for an appropriate choice of $A$, it takes the form 
\[
\naka(\vi_1,\w_1)\times \naka(\vi_2,\w_2)
\]
with $\vi_1+\vi_2=\vi$ and $\w_1+\w_2=\w$. In this particular case, the stable envelope becomes a collection of maps 
\begin{equation}
	\label{stable envelope map intro}
	\text{Stab}_{\mathfrak{C}}:H_{T_{\w_1}}(\naka(\vi_1,\w_1))\otimes_\Bbbk H_{T_{\w_2}}(\naka(\vi_2,\w_2))\to H_{T_{\w_1+\w_2}}(\naka(\vi_1+\vi_2,\w_1+\w_2))
\end{equation}

\subsubsection{R-matrices and Yangians}

Stable envelopes of Nakajima varieties give rise to R-matrices \cite[Chapter 4]{maulik2012quantum}. We now briefly review their construction.
Taking direct sums of \eqref{stable envelope map intro} with respect to $\vi_1$ and $\vi_2$ and $\vi$, one gets a map (which we denote in the same way)
\begin{equation}
	\text{Stab}_{\mathfrak{C}}:H_{T_{\w_1}}(\naka(\w_1))\otimes_\Bbbk H_{T_{\w_2}}(\naka(\w_2))\to H_{T_{\w_1+\w_2}}(\naka(\w_1+\w_2))
\end{equation}
where $\naka(\w_i)=\bigsqcup_{\vi_i} \naka(\vi_i,\w_i)$. Given a pair of different chambers $\mathfrak{C_1}$ and $\mathfrak{C_2}$, the composition
\[
R_{\mathfrak{C_1}, \mathfrak{C_2}}:=\text{Stab}_{\mathfrak{C_1}}^{-1}\circ \text{Stab}_{\mathfrak{C_2}}\in \End(H_{T_{\w_1}}(\naka(\w_1))_{loc}\otimes_\Bbbk H_{T_{\w_2}}(\naka(\w_2))_{loc})
\]
is called geometric R-matrix. Here, the subscript \emph{loc} stands for localization with respect to the equivariant parameters coming from the framing. In other words, we set 
\[
H_{T_{\w}}(-)_{loc}=H_{T_{\w}}(-)\otimes_{\Bbbk}\text{Frac}(H_{A_{\w}}(\pt)).
\]
Localization is required to define the inverse $ \text{Stab}_{\mathfrak{C_1}}^{-1}$ of a stable envelope.

A foundational result of Maulik-Okounkov's theory is that these R-matrices satisfy the Yang-Baxter equation with spectral parameter. The associated quantum group, defined via RTT relations, is the Maulik-Okounkov's Yangian $Y^Q_{MO}$ associated to the quiver $Q$ \cite[Chapter 5]{maulik2012quantum}. 
The subalgebra of $Y^Q_{MO}$ spanned by the matrix elements of the classical r-matrix $r$, defined as the equivariant residue of
\[
R_{\mathfrak{C_1}, \mathfrak{C_2}}(a)=1+ \frac{\hbar}{a} r+ O(a^{-2})
\]
is the Maulik-Okounkov's Lie algebra $\g_Q$.

\subsubsection{Framed preprojective semistable CoHA}

The apparent similarity with \eqref{CoHA multiplication maps intro} suggests tentatively defining an algebra structure on
\[
\HhQ=\bigoplus_{\vi,\w\in \N^I}\HhQ(\vi,\w) \qquad\qquad \HhQ(\vi,\w):=H_{T_\w}(\naka(\vi,\w))
\]
via the stable envelopes maps \eqref{stable envelope map intro}. Indeed, this construction defines a graded associative algebra structure on $\HhQ$, which we call the framed semistable preprojective cohomological Hall algebra of $Q$. The choice of the name can be partially understood by observing that the vector space $\HhQ$ consists of cohomologies of Nakajima varieties, which are exactly moduli spaces of semistable representations of the preprojective path algebra $\Ci \overline Q_{\text{fr}}/ J_\mu$ of the doubled framed quiver $\overline Q_{\text{fr}}$, but of course, more is needed to justify this definition. Namely, one must show that the multiplication defined by stable envelopes has something to do with CoHA's multiplication. This is the main goal of the article.
To accomplish this, we exploit Aganagic and Okounkov's non-abelian stable envelope \cite{aganagic2017quasimap}, a map of $ H_{T_\w}(\pt)$-modules
\begin{equation}
    \label{nonabelian stable envelope}
    \bm\psi: H_{T_\w}(\naka(\vi,\w))\to H_{T_\w}([T^*\Rep(\vi,\w)/G_\vi])
\end{equation}
that solves, in analogy with the ordinary stable envelopes, an interpolation problem for cohomology classes \cite{okounkov2020inductiveII}. In a nutshell, the assignment $\alpha\mapsto \bm\psi(\alpha)$ can be thought of as a distinguished extension of the cohomology class $\alpha$ from $\naka(\vi,\w)$ to the whole stack of representations $[T^*\Rep(\vi,\w)/ G_{\vi}]$. The extension $\bm\psi(\alpha)$ is characterized by a precise support condition, namely being supported on $[\mu_{\vi,\w}^{-1}(0)/G_{\vi}]\subset [T^*\Rep(\vi,\w)/ G_{\vi}]$. As a result, $\bm\psi$ is, up to a multiplicative factor, a section of the tautological pullback 
\[
\bm{j}^*: H_{T_\w}([T^*\Rep(\vi,\w)/ G_{\vi}])\twoheadrightarrow H_{T_\w}(\naka(\vi,\w))
\]
associated with the inclusion
\[
\bm{j}:\naka(\vi,\w)= \mu^{-1}_{\vi,\w}(0)^{ss}/G_{\vi}\hookrightarrow [T^*\Rep(\vi,\w)/ G_{\vi}].
\]

With this background, we can finally compare the algebra structures on $\HhQ$ and $\CoHAQ$. Firstly, consider the map $\bm\psi: \HhQ\to \CoHAQ$ obtained by taking direct sums of \eqref{nonabelian stable envelope}. The main result of the article is the following:
\begin{thm*}[\ref{main theorem}]
	The map $\bm\psi: \HhQ\to \CoHAQ_{\tau}$ is an injective morphism of graded algebras. 
\end{thm*}
The subscript $\tau$ in the notation refers to the fact that the algebra structure on $\CoHAQ$ is appropriately twisted, although in a simple and geometrically meaningful way. In words, the theorem above says that the map $\bm\psi$ realizes the cohomology of Nakajima varieties as a subalgebra of the framed CoHA $\CoHAQ_{\tau}$, and on this subalgebra, CoHA's multiplication is identified with taking stable envelopes. 

\subsubsection{}

An an interesting application of the previous theorem, we produce an explicit inductive formula for the stable envelopes of Nakajima varieties. The map $\bm\psi$ allows giving a tautological presentation of the stable envelopes by setting
\[
\text{Stab}_{\mathfrak{C}}^{\bm\psi}:=\frac{\bm{\psi}\circ \text{Stab}}{e({\hbar}\g_\vi)}.
\]
This map has as target a localization of the tautological ring $H_{T_\w}([T^*\Rep(\vi,\w)/G_\vi])$ and it restricts to $\text{Stab}_{\mathfrak{C}}$ on $\naka(\vi,\w)\subset [T^*\Rep(\vi,\w)/G_\vi]$.  

Applying Theorem $\ref{main theorem}$, we show that for every fixed component $F\subset \naka(\vi,\w)^A$ and chamber $\mathfrak{C}$, there are $k=\text{dim}(A)$ decompositions $\vi=\vi_1+\vi_2$, $\w=\w_1+\w_2$, fixed components $F_1\subset \naka(\vi_1,\w_1)^{A_1}$, $F_2\subset \naka(\vi_2,\w_2)^{A_2}$ and chambers $\mathfrak{C}^1$, $\mathfrak{C}^2$ such that $A=A_1\times A_2$ and 
\[
\text{Stab}^{\bm\psi}_{\mathfrak{C}}(F)=\Shuffle\left( \frac{e(T^*\Rep(\vi,\w)[-1]) \text{Stab}^{\bm\psi}_{\mathfrak{C}^1}(F_1) \text{Stab}^{\bm\psi}_{\mathfrak{C}^2}(F _2)}{e(\g_{\vi}[-1]) e({\hbar}\g_{\vi}[-1])} \right).
\]
All the classes in the formula have an explicit tautological presentation.
By induction on the dimension of $A$, this formula gives an explicit presentation of the stable envelopes of Nakajima varieties for the action of a torus $A\subset A_\w$. This formula can be seen as the cohomological limit of the main result of \cite{botta2021shuffle}, where an analogous formula is presented in the elliptic setting.

\subsection{Further directions}

This article can be seen as the starting point of a larger project aimed at establishing a connection between the geometric theory of quantum groups initiated with the introduction of stable envelopes and cohomological Hall algebras. We conclude this introduction with a brief discussion of the prospective directions of development of the author's current work.

\subsubsection{K-theory, elliptic cohomology and qKZB equations}
First of all, the results of the present paper will be extended to K-theory and elliptic cohomology in the companion paper \cite{botta2022CoHA2}. While the results in K-theory are very close to the cohomological ones and can be essentially guessed from this article, the elliptic setting requires a more sophisticated approach. This is partly due to features intrinsic to elliptic cohomology and partly because the elliptic stable envelopes depend on extra parameters called dynamical (or Kähler) parameters. In view of this, the algebras of this article will be replaced by certain algebra objects in an appropriate monoidal category, whose tensor product involves some shifts of the dynamical parameters. Because of that, this construction can be seen as a dynamical version of the sheafified elliptic cohomological Hall algebra of Yang and Zhao \cite{yang2017ell}.

An important application of the CoHA's structure of stable envelopes is in the world of difference equations. Integral solutions for the quantum Knizhnik–Zamolodchikov equations (qKZ) associated with arbitrary quivers can be expressed as 
\[
\int_{\gamma\subset G_\vi}  \text{Stab}^K \Gamma_q,
\]
where $\Gamma_q$ is a product of $q$-gamma functions and $\text{Stab}^K $ is a K-theoretic stable envelope. The proof of this result is due to Okounkov and his coauthors \cite{okounkovsmirnov2016Kthy, okounkov2017enumerative, aganagic2017quasimap}, and is based on the equivalence of certain enumerative counts in K-theory. 
As these methods cannot be easily transferred to elliptic cohomology, the problem of finding solutions for the quantum Knizhnik–Zamolodchikov-Bernard equations (qKZB), which are the elliptic incarnation of the qKZ equations, is still open for general quivers.

In a joint work of the author with Giovanni Felder and Keyu Wang, the CoHA's algebras structure of the elliptic stable envelopes will be exploited to produce such solutions. This project widely generalizes the results of Felder, Tarasov, and Varchenko \cite{felder1999resonance, felder1996solutions}, who had already obtained solutions for type A quivers. 

\subsubsection{Cohomological Hall algebras and stable envelopes of bow varieties}

The elliptic stable envelope of a Nakajima variety $X$ can be interpreted as the transformation matrix that relates the vertex functions of $X$, which are certain enumerative invariants that provide q-holonomic modules for a wide class of difference equations such as the q-KZ equations \cite{aganagic2016elliptic}, with the vertex functions of its 3d mirror dual $X^!$. Consequently, stable envelopes can be considered an adequate topological invariant to test 3d mirror symmetry. Some positive results have already been obtained in \cite{Rimanyi_2019full, Rimanyi2019grass}, but only in the fortunate case when the 3d-dual of the Nakajima variety $X$ is itself a Nakajima variety. Indeed, the class of Nakajima varieties is not closed under mirror symmetry, so to work in full generality, one needs to consider a larger class, namely bow varieties, which have been introduced in the context of enumerative geometry in \cite{rimanyi2020bow}. In joint work with Rich\'ard Rim\'anyi \cite{botta_rimanyiCoHA}, it will be shown that all the main ideas of this paper, namely stable envelopes, cohomological Hall algebras, and their interplay can be extended to bow varieties. Moreover, these results will be exploited to prove fusion rules for R-matrices and 3d mirror symmetry of elliptic stable envelopes \cite{botta_rimanyiMirror}. 

\subsubsection{Preprojective CoHA and Yangians}

Going back to the cohomological level, another version of CoHA can be related to those discussed in this article. Namely,  one can consider the preprojective version of \eqref{unframed CoHA} by setting \[
\pCoHA^{\overline Q}=\bigoplus_{\vi\in \N^I} \pCoHA^{\overline Q}(\vi)\qquad\qquad  \pCoHA^{\overline Q}(\vi)=H^\text{{BM}}([\mu_{\vi}^{-1}(0)/G_\vi]),
\]
where $H^\text{{BM}}$ stands for Borel-Moore homology.
As shown in \cite{yang2018CoHA}, $\pCoHA^{\overline Q}$ acts on cohomology of Nakajima varieties by maps
\[
\pCoHA^{\overline Q}(\vi_1)\otimes_\Bbbk \HhQ(\vi_2,\w)\to \HhQ(\vi_1+\vi_2,\w).
\]

Davison in \cite{Davison_critical_CoHA} and Schiffmann and Vasserot in \cite{Schiffmann2017generators} conjectured in that $\pCoHA^{\overline Q}$ is isomorphic to $Y_{MO}^{Q,+}$, the positive half of the Maulik-Okounkov Yangian $Y_{MO}^{Q}$. This was first shown for the 1-loop quiver in \cite{schiffmann2012cherednik, maulik2012quantum}.
In parallel, Yang and Zhao \cite{yang2018CoHA, yang2016affine, yang2017ell} showed that an appropriate subalgebra of $\pCoHA^{\overline Q}$ is isomorphic to the positive half of Drinfel'd Yangians of type ADE. 

More recently, Schiffmann and Vasserot \cite{Schiffmann2017generators, Schiffmann2017yangians} proved half the conjecture by building an injective map from $\pCoHA^{\overline Q}\otimes_\Bbbk \text{Frac}(\Bbbk)$ to $Y_{MO}^{Q,+} \otimes \text{Frac}(\Bbbk)$, which intertwines the actions of these algebras on the cohomology of Nakajima varieties. The key technical step in their work consisted in identifying certain generators of $\pCoHA^{\overline Q}\otimes_\Bbbk \text{Frac}(\Bbbk)$, expressed in terms of Hecke correspondences acting on the cohomology of Nakajima varieties, with the matrix coefficients of the classical r-matrix of $Y_{MO}^{Q,+}$, which in turn is generated by stable envelopes. In other words, Schliffman and Vasserot identify the appropriate pull-push diagrams of certain matrix coefficients of the equivariant residue of the stable envelope.

The main result of the present paper goes in a similar but independent direction by showing that the whole stable envelope comes from the pull-push diagram defining the multiplication map \eqref{CoHA multiplication maps intro} of the framed CoHA $\CoHAQ_{\tau}$.

Since $\pCoHA^{\overline Q}$ admits a canonical algebra morphism to $\CoHA^{\overline Q}\subset\CoHA^{\overline Q_{\text{fr}}}$, we expect Schliffman-Vasserot's result to be compatible with ours, and hope that a positive answer to this prediction would allow shedding more light on the connection between the preprojective CoHA and Maulik-Okounkov's Yangians $Y_{MO}^{Q,+}$.

\subsubsection{Elliptic quantum groups}

Of course, preprojective CoHAs can be also constructed in K-theory and elliptic cohomology \cite{padurariu_KHA, yang2017ell}. On the other hand, a definition of the elliptic quantum group of an arbitrary quiver is missing in the literature. As a consequence, an identification of the preprojective CoHA with the positive half of quantum groups in cohomology and K-theory could be used to define the elliptic quantum group itself as a Drinfel'd double of the elliptic preprojective CoHA.

\subsection{Acknowledgments} 

I want to thank my advisor Giovanni Felder for his guidance and constant attention throughout the development of this project. I also thank Richard Rimanyi for drawing my attention to CoHAs and Stefano D'Alesio for useful conversations.

This work was supported as a part of NCCR SwissMAP, a National
Centre of Competence in Research, funded by the Swiss
National Science Foundation (grant number 205607)
and grant $200021\_196892$ of the Swiss National
Science Foundation.

\subsection{
	Declarations}
\subsubsection{Conflict of interest}
The author states that there is no conflict of interest.

\subsubsection{Data Availability}
Data sharing is not applicable to this article as no datasets were generated or analysed during the current study.


\section{Set-up}

\subsection{Cohomology functors}

\subsubsection{}

Let $G$ be an algebraic group. The cohomology functor that we consider in this article is equivariant cohomology
\[
H_G(-): G\text{-Spaces}\to \text{Alg}_{H_G(\pt)},
\]
taken with complex coefficients. 
Given a $G$-space $X$, its $G$-equivariant cohomology $H_G(X)$ is defined as the ordinary cohomology of the space 
\[
X_G= (X\times EG)/G,
\]
where $EG\to BG$ is the classifying space of $G$. By construction, it is an $H_G(\pt)=H(BG)$-algebra.

\subsubsection{}

In this article, we only consider the case when $G$ is either a torus $T$, a parabolic subgroup $P$ of $\GL(n)$, or $\GL(n)$ itself. In these cases, $EG$ and $BG$ admit explicit representatives. The infinite Grassmannian 
\[
\Gr{n}{\infty }=\lim_{k\to\infty } \Gr{k}{\Ci^n}
\]
is a model for the classifying space of $B\GL(v)$ with universal bundle given by its tautological bundle of rank $n$. Similarly, a model for the classifying space of a $n$-dimensional torus $T\cong (\Ci^\times)^n$ is $(\mathbb{P}^\infty)^n$, with $\mathbb{P}^\infty=\Gr{1}{\infty}$, and the universal bundle $ET$ is the direct sum of the tautological line bundles living on each copy of $\mathbb{P}^\infty$.
As a consequence, we have 
\[
H_{T}(\pt)= \Ci[\ti]\qquad H_{\GL(n)}(\pt)=\Ci[\ti]^W,
\]
where $\mathfrak{t}$ is the Lie algebra of $T$ and $W$ is the Weyl group of $\GL(n)$.
Let now $P$ be a parabolic subgroup of $\GL(n)$ containing the Levi subgroup $L\cong \prod_i GL(n_i)$. Its classifying space fits in the Cartesian squares
\[
\begin{tikzcd}
 ET\arrow[r]\arrow[d] & EP\arrow[r]\arrow[d] & E\GL(n)\arrow[d]\\
 BL=\prod_i \Gr{n_i}{\infty }\arrow[r] & BP\arrow[r] & B\GL(n)=\Gr{n}{\infty }
\end{tikzcd}
\]
Additionally, notice that $BP$ is homotopy equivalent to $BL$ and a fibration over $BGL(n)$ by partial flags $\GL(n)/P$.

\subsubsection{}

Beyond the standard pullback in cohomology, there are also pushforward maps 
\[
f_*: H_G(X)\to H_G(Y)
\]
associated with proper equivariant morphisms between smooth varieties $f: X\to Y$. Their construction is standard and is obtained through Poincar\'e duality in Borel-Moore homology.

\subsubsection{}

In the article, we will also consider the cohomology of quotients stacks $[X/G]$, which is the algebraic counterpart of equivariant cohomology. As discussed in \cite{behrend}, the cohomology of $[X/G]$ coincides with the equivariant cohomology of the prequotient:
\[
H([X/G])=H_G(X).
\]
More generally, if we are given an action of $G\times H$ on $X$, we can define the $H$-equivariant cohomology of $[X/G]$ as
\[
H_{H}([X/G])=H([X/(G\times H)]).
\]

\subsection{Stacks of quiver representations}
\subsubsection{}
Let $(Q, I)$ be a quiver, i.e. an oriented graph, with finite vertex set $I$, where edge-loops and multiple edges are allowed. The quiver data is encoded in the adjacency matrix $Q=\lbrace Q_{ij} \rbrace$,
where
\[
Q_{ij}=\lbrace\text{number of edges from $i$ to $j$}\rbrace.
\]
By abuse of notation, we identify the set of oriented edges with the matrix $Q$.
There are two natural ``head" and ``tail" maps $h,t:Q\to I$ that associate with an oriented edge $e\in Q$ its head and tail vertices $h(e)$ and $t(e)$ respectively.

A representation of a quiver$(Q,I)$ is an assignment of a vector space $V_i$ to every vertex $i\in I$ and a linear map $x_{h(e),t(e)}\in \Hom(V_{t(e)},V_{h(e)})$ to every edge $e\in Q$. The space of representations $\Rep(\vi)$ is
\begin{equation*}
\label{lin repn}
\Rep(\vi)=\bigoplus_{e\in Q}\Hom(V_{(t(e))}, V_{(h(e))}).
\end{equation*}
where $V_i=\Ci^{\vi_i}$.

\subsubsection{}

Given a quiver $(Q, I)$, we also introduce the space of framed representations
\[
\Rep(\vi,\w)\cong \Rep(\vi)\oplus \bigoplus_{i\in I}\Hom(V_i, W_i)
\]
with dimension
\[
\vi_i=\dim(V_i)\qquad \w_i=\dim(W_i), \qquad \vi,\w\in \N^I.
\]
Equivalently, this can be seen as the space of representations of the framed quiver $( Q_{\text{fr}}, I_{\text{fr}})$, which is the quiver whose set of vertices $ I_{\text{fr}}$ consists of two copies of $I$ and whose adjacency matrix is 
\begin{equation*}
\label{framed dubled quiver}
Q_\text{fr}=\begin{pmatrix}
Q & 0\\
1 & 0
\end{pmatrix}.
\end{equation*}
Also notice that the cotangent bundle $T^*\Rep(\vi,\w)$ can be identified, via trace pairing, with the representations of the framed doubled quiver $(\overline{Q_\text{fr}}, \overline I_{\text{fr}})$, defined by $\overline I_\text{fr}=I_\text{fr}$ and $\overline{Q}_\text{fr}=Q_\text{fr}+Q_\text{fr}^T$.

Finally, we remark that $\Rep(\vi,0)$ is naturally identified with $\Rep(\vi)$ and, similarly, $T^*\Rep(\vi,0)$ is identified with $T^*\Rep(\vi)$, the space of representations of the (unframed) doubled quiver $(\overline{Q},\overline{I})=(Q+Q^T, I)$.

\subsubsection{}
The groups
\[
G_{\vi}:=\prod_{i\in I}\GL(\vi_i)\qquad G_{\w}:=\prod_{i\in I}\GL(\w_i)
\]
act naturally on 
on $T^*\Rep(\vi,\w)$ by change of basis. There is also an action of $\Ci_{\hbar}^\times$ rescaling the cotangent directions of $T^*\Rep(\vi,\w)$, and we denote the corresponding weight by ${\hbar}$.
Overall, we get an action of $G_\vi\times G_\w\times \Ci^\times_{\hbar}$ on $T^*\Rep(\vi,\w)$.

Notice that the vector space $T^*\Rep(\vi,\w)$ admits a canonical symplectic structure, which is preserved by the action of $G_\vi\times T_\w$ and is shifted by $\Ci^\times_{\hbar}$ with weight ${\hbar^{-1}}$.
\begin{defn}
The moduli space $\RG(\vi,\w)$ is the quotient stack $[T^*\Rep(\vi,\w)/G_\vi]$.
\end{defn}
Since we will never consider the stack $[\Rep(\vi,\w)/G_\vi]$, no confusion should arise from the notation.

Notice that we only take the quotient with respect to $G_\vi$, so $\RG(\vi,\w)$ admits a residual action by $G_\w\times \Ci^\times_{\hbar}$. We will be actually interested in the action of a maximal torus $T_\w\subset G_\w\times \Ci^\times_{\hbar}$, which we factor as $T_\w=A_\w \times  \Ci^\times_{\hbar}$
where $A_\w=\prod_{i\in I} A_{\w_i}$ is a maximal torus in $G_\w$.


\subsection{Nakajima quiver varieties}

\subsubsection{}
By $\naka_{Q,\theta}(\vi,\w)$ we denote the Nakajima variety associated to a quiver $(Q,I)$, dimension vectors $\vi,\w\in \N^I$ and stability condition $\theta\in\Z^{I}\cong \text{Char}(G_\vi)$, cf. \cite{ginzburg2009lectures, Nakajimainstantons, Nakajimaquiver, Nakajima_2001}.
Namely, $\naka_{Q,\theta}(\vi,\w)$ is realized as the GIT symplectic reduction of the space of representations $T^*\Rep(\vi,\w)$ by the Hamiltonian action of $G_{\vi}$:
\[
\naka_{Q,\theta}(\vi,\w):=T^*\Rep(\vi,\w)////^\theta G_{\vi}.
\]
Equivalently, $\naka_{Q,\theta}(\vi,\w)$ is the quotient 
\[
\mu_{\vi,\w}^{-1}(0)^{\theta-ss}/G_\vi,
\]
where 
\[
\mu_{\vi,\w}:T^*\Rep(\vi,\w)\to \g_\vi^*
\] is the moment map of the $G_{\vi}$-action and $\mu_{\vi,\w}^{-1}(0)^{\theta-ss}$ is the open subset of $\theta$-semistable representations in $\mu_{\vi,\w}^{-1}(0)$. Since the action of $T_\w=A_\w\times \Ci^\times_{\hbar}$ commutes with the one of $G_{\vi}$, it descends to an action on $\naka_{\theta}(\vi,\w)$. In particular, the subgroup $\Ci^\times_{\hbar}$ rescales the symplectic form while all the other actions preserve it. 

\subsubsection{}

A striking feature of Nakajima varieties is that the fixed components of the action of a subtorus of $A_\w$ are still quiver varieties \cite[Section 2.4.1]{maulik2012quantum}.
\begin{propn}
\label{propoistion fixed point naka}
Let $\Ci^\times$ be a one-dimensional subtorus of $A_\w$ acting trivially on the subspaces $W^{(1)}_i\subset W_i$, and with weight one on their complements $W^{(2)}_i$. Then
\begin{equation*}
    \naka(\vi,\w)^{\Ci^\times}\cong\bigsqcup_{\vi_1+\vi_2=\vi}\naka(\vi_1,\w_1)\times \naka(\vi_2,\w_2),
\end{equation*}
where $\w_1$ and $\w_2$ are the dimension vectors of $W^{(1)}$ and $W^{(2)}$ and the isomorphism is induced by the diagonal inclusion
\[
T^*\Rep(\vi_1,\w_1)\times T^*\Rep(\vi_2,\w_2)\hookrightarrow T^*\Rep(\vi,\w).
\]
\end{propn}
By iteration, one gets a decomposition of the connected components of $\naka(\vi,\w)^{A_\w}$ as a product of Nakajima varieties.

\subsubsection{}
\label{kirwan surjectivity}

A representation $U$ of $G_\vi\times T_\w$ allows defining the vector bundle
\begin{equation}
\label{tautological bundles naka}
    (\mu_{\vi,\w}^{-1}(0)^{\theta-ss}\times U)/G_{\vi}
\end{equation}
over $X$, which, by abuse of notation, we denote in the same way. In particular, this gives tautological bundles $V_i$ and $W_i$ for all $i\in I$.

Kirwan surjectivity, proved by McGerty and Nevins in \cite{McGerty_2017} and by Schiffmann and Vasseort in an appendix of \cite{Schiffmann2017generators}, states that the restriction map 
\[
H_{T_\w\times G_\vi}(T^*\Rep(\vi,\w))\to H_{T_\w\times G_\vi}(\mu_{\vi,\w}^{-1}(0)^{\theta-ss})=H_{T_\w}(\naka(\vi,\w))
\]
is surjective, and, as a consequence, the cohomology ring $H_{T_\w}(\naka(\vi,\w))$ is generated as a $H_{\Ci_\hbar^\times }(\pt)$-algebra by the Chern roots of the tautological bundles $V_i$ and $W_i$.



\section{Framed CoHA}

\subsection{The basic construction}

\subsubsection{}

Fix a quiver $(Q,I)$ and consider once again a pair of decompositions $\vi=\vi_1+\vi_2$ and $\w=\w_1+\w_2$ and the inclusion 
\begin{equation}
    \label{diagonal inclusion representations}
    T^*\Rep(\vi_1,\w_1)\times T^*\Rep(\vi_2,\w_2)\hookrightarrow T^*\Rep(\vi,\w).
\end{equation}
For $j=1,2$, set 
\[
V^{(j)}=\bigoplus_{i\in I} V_i^{(j)}\qquad W^{(j)}=\bigoplus_{i\in I} W_i^{(j)}
\]
where $\vi_j=(\dim V_i^{(j)})_{i\in I}$ and $\w_j=(\dim W_i^{(j)})_{i\in I}$.
Interpreting the decompositions above as the introduction of an additional grading
\[
V=V^{(1)}[0]\oplus V^{(2)}[1] \qquad W=W^{(1)}[0]\oplus W^{(2)}[1],
\]
we can identify the image of \eqref{diagonal inclusion representations} with the degree zero subspace in the induced grading
\[
T^*\Rep(\vi,\w)=T^*\Rep(\vi,\w)[-1]\oplus T^*\Rep(\vi,\w)[0]\oplus T^*\Rep(\vi,\w)[1].
\]
The space
\[
Z:=T^*\Rep(\vi,\w)[0]\oplus T^*\Rep(\vi,\w)[1]
\]
naturally lives on a Lagrangian correspondence 
\begin{equation}
\label{fundamental correspondence prequotient}
    \begin{tikzcd}
    T^*\Rep(\vi_1,\w_1)\times T^*\Rep(\vi_2,\w_2) &  Z\arrow[l] \arrow[r, hook]  & T^*\Rep(\vi,\w)
    \end{tikzcd}
\end{equation}
where the first map is projection and the second one is inclusion. By Lagrangian, here we mean that $Z$ can be seen as a Lagrangian subspace of the symplectic space
\[
\left(T^*\Rep(\vi_1,\w_1)\times T^*\Rep(\vi_2,\w_2)\right)\times T^*\Rep(\vi,\w)
\]
with symplectic form given by the sum of the standard symplectic form of the first factor and the opposite of the standard symplectic form of the second one. 
\begin{remq}
The previous grading is equivalent to an action of a one-dimensional torus $\Ci^\times$ acting on the various spaces with the weights prescribed by their degrees. From this point of view, the zero-degree subspace 
\[
T^*\Rep(\vi,\w)[0]\cong T^*\Rep(\vi_1,\w_1)\times T^*\Rep(\vi_2,\w_2)
\]
is just the fixed locus of the action of $\Ci^\times$ on $T^*\Rep(\vi,\w)$ and $Z$ is the attracting set of $T^*\Rep(\vi,\w)[0]$.
\end{remq}

\subsubsection{}

It is important to note that
\[
p: Z\to T^*\Rep(\vi_1,\w_1)\times T^*\Rep(\vi_2,\w_2)
\]
is a vector bundle whose fibers $Z_{(x_1,x_2)}$ can be identified with the ext groups
\[
\text{Ext}^1(x_2,x_1).
\]
Here we see the elements $x_1$ and $x_2$ as modules over the path algebra of the framed doubled quiver $(\overline{Q}_\text{fr}, \overline I_{\text{fr}})$ and the Ext-group is taken in this category.

\subsubsection{}

The subspace $Z\subset T^*\Rep(\vi,\w)$ is preserved by the action of
\[
P=\prod_{i\in I} P_i \qquad  P_i=\begin{pmatrix}G_{\vi_1} & 0 \\ \ast & G_{\vi_2}\end{pmatrix}
\]
i.e. the product of the parabolic subgroups preserving $V_i^{(2)}\subset V_i$. Notice in particular that
\[
\text{Lie}(P)=\g_{\vi_1}\oplus \g_{\vi_2}\oplus \g_{\vi}[1].
\]
Passing to quotients, we get a correspondence
\begin{equation}
\label{fundamental correspondence stacks}
    \begin{tikzcd}
    \RG(\vi_1,\w_1)\times\RG(\vi_2,\w_2) &   \left[Z/P\right]\arrow[l, ,swap, "\mathsf{p}"] \arrow[r, "\mathsf{q}"] & \RG(\vi,\w)
    \end{tikzcd}
\end{equation}

Following the general construction of Kontsevich and Soibelman, we will define an algebra via pullback-pushforward along these maps.

\subsubsection{}
\label{Subsubsection tangent and normal classes}
For later reference, we record some formulas for certain K-theory classes of $\RG(\vi,\w)$ and $\naka(\vi,\w)$. Firstly, the tangent class
\[
T\RG(\vi,\w)\in K_{T_\w}(\RG(\vi,\w))= K_{T_\w\times G_\vi}(\pt)
\] 
can be written as 
\[
T^*\Rep(\vi,\w)-\g_\vi.
\]
The inclusion \eqref{diagonal inclusion representations} descends a $T_\w$-equivariant closed immersion of stacks 
\[
\RG(\vi_1,\w_1)\times\RG(\vi_2,\w_2)\hookrightarrow\RG(\vi,\w)
\]
whose normal class
\[
\mathfrak{N}=T\RG(\vi,\w)-T\RG(\vi_1,\w_1)-T\RG(\vi_2,\w_2)\in K_{T_\w}(\RG(\vi_1,\w_1)\times \RG(\vi_2,\w_2))
\]
can be rewritten as $\mathfrak{N}=\mathfrak{N}[1]+\mathfrak{N}[-1]$, where 
\[
\mathfrak{N}[1]=T^*\Rep(\vi,\w)[1] -\g_\vi[1] \qquad\mathfrak{N}[-1]= T^*\Rep(\vi,\w)[-1] -\g_\vi[-1].
\]
Similarly, we have 
\[
T\naka(\vi,\w)=T^*\Rep(\vi,\w)-\g_\vi-{\hbar}\g_\vi\in K_{T_\w}(\naka(\vi,\w))
\]
and the normal class of 
\[
\naka(\vi_1,\w_1)\times\naka(\vi_2,\w_2)\hookrightarrow\naka(\vi,\w)
\]
is $N=N[1]+N[-1]$, where 
\[
N[1]=T^*\Rep(\vi,\w)[1] -\g_\vi[1]-{\hbar}\g_\vi[1] \qquad N[-1]= T^*\Rep(\vi,\w)[-1] -\g_\vi[-1]-{\hbar}\g_\vi[-1].
\]
\subsection{Framed Cohomological Hall Algebras}

\subsubsection{}

From now on, we denote the $\Ci^\times_{\hbar}$-equivariant cohomology of the point by $\Bbbk$. This means that $\Bbbk$ is the polynomial ring $\Ci[{\hbar}]$. Let now $(Q, I)$ be a quiver. We define a $\N^I\times \N^I$-graded $\Bbbk$-module
\[
\CoHAQ=\bigoplus_{\vi,\w\in \N^I} \CoHAQ(\vi,\w)
\]
by setting 
\[
\CoHAQ(\vi,\w):=H_{T_\w}(\RG(\vi,\w)).
\]
As a $\Bbbk$-module, $\CoHAQ(\vi,\w)$ is just the ring
\[
\Bbbk[\liea_\w\times \s_{\vi}]^{W_{G_{\vi}}},
\]
where $\liea_\w$ and $\s_\vi$ are the Lie algebras of $A_\w\subset T_\w$ and of some maximal torus $S_\vi\subset G_\vi$ respectively.
Whenever convenient, we will drop the reference to the quiver by simply writing $\CoHA$ instead of $\CoHAQ$.
Clearly, instead of defining $\CoHA(\vi,\w)$ as the $T_\w$-equivariat cohomology of $\RG(\vi,\w)=[T^*\Rep(\vi,\w)/G_\vi]$, we could have taken the non-equivariant cohomology of the stack $[T^*\Rep(\vi,\w)/(G_\vi\times T_\w)]$. However, we believe that the former is the right perspective, especially when it comes to studying the relation between stable envelopes and CoHAs.

Following Kontsevich and Soibelman \cite{kontsevich2011cohomological}, we define a multiplication map
\[
\mathsf{m}:\CoHA\otimes_{\Bbbk}\CoHA\to \CoHA
\]
that turns $\mathcal{H}$ into a $\N^I\times \N^I$-graded $\Bbbk$-algebra. Each graded component 
\[
\mathsf{m}:\CoHA(\vi_1,\w_1)\otimes_{\Bbbk}\CoHA(\vi_2,\w_2)\to \CoHA(\vi_1+\vi_2,\w_1+\w_2)
\]
is defined by means of correspondence \eqref{fundamental correspondence stacks}. Explicitly, $\mathsf{m}$ is defined as the composition of the following maps:
\begin{enumerate}
    \item the K\"unneth isomorphism
    \[
    H_{T_{\w_1}}^*(\RG(\vi_1,\w_1))\otimes_{\Bbbk}H_{T_{\w_2}}^*(\RG(\vi_2,\w_2))\cong H_{T_\w}^*(\RG(\vi_1,\w_1)\times \RG(\vi_2,\w_2)),
    \]
    \item the pullback
    \[
    \mathsf{p}^*: H_{T_\w}^*(\RG(\vi_1,\w_1)\times \RG(\vi_2,\w_2))\to H_{T_\w}([Z/P]),
    \]
    \item the pushforward
    \[
    \mathsf{q}_*: H_{T_\w}([Z/P])\to H_{T_\w}(\RG(\vi_1+\vi_2,\w_1+\w_2)).
    \]
\end{enumerate}
The proof of \cite[Theorem 1]{kontsevich2011cohomological} extends to our framed setting and gives the following

\begin{thm}
\label{theorem CoHA is an algebra}
The map $\mathsf{m}$ defines an $\N^I\times \N^I$-graded associative unital $\Bbbk$-algebra structure on $\CoHAQ$ with unit element given by $1\in \Bbbk=\CoHAQ(0,0)$.
\end{thm}
\begin{defn}
We call the associative unital $\N^I\times \N^I$-graded algebra $(\CoHAQ,\mathsf{m})$ the framed doubled cohomological Hall algebra associated with the quiver $(Q, I)$.
\end{defn}
Notice that, seen as a $\Bbbk$-module, $\CoHAQ$ does not depend on the edges of $Q$, encoded in the multiplication map $\mathsf{m}$ nonetheless.
\begin{remq}
The (unframed) CoHA associated with the doubled quiver $(\overline Q, \overline I)$ defined by Kontsevich-Soibelman is recovered as the $\w=0$ subalgebra of our framed CoHA $\CoHAQ$. Indeed, the $\Bbbk$-submodule with trivial $\w$-grading is preserved by the multiplication, which reduces to the one defined in \cite[Section 2.2]{kontsevich2011cohomological}. On the other hand, the algebra $\CoHAQ$ should not be confused with Kontsevich-Soiblelman cohomological Hall algebra associated with the framed doubled quiver $(\overline Q_{\text{fr}},\overline I_{\text{fr}})$. Indeed, the graded component of the latter is identified with $H^*_{G_{\vi}\times G_{\w}\times \Ci^\times_{\hbar}}(T^*\Rep(\vi,\w))$ rather than $H^*_{G_{\vi}\times T_{\w}}(T^*\Rep(\vi,\w))$, and also the multiplication is affected by the change of group. 
\end{remq}

\subsubsection{}

To relate $\CoHAQ$ with stable envelopes, we need to twist the multiplication
\[
\mathsf{m}:\CoHA(\vi_1,\w_1)\otimes_\Bbbk\CoHA(\vi_2,\w_2)\to \CoHA(\vi_1+\vi_2,\w_1+\w_2).
\]
This is done as follows: multiplication by the class
\[
e({\hbar}\g_\vi[1])\in H_{T_\w}^*(\RG(\vi_1,\w_1)\times \RG(\vi_2,\w_2))
\]
gives an endomorphism that can be inserted between the maps at points 1. and 2. in the definition of $\mathsf{m}$ to get a twisted multiplication $\mathsf{m}_\tau$. It is easy to see that $\mathsf{m}_\tau$ also defines unital associative $\Bbbk$-algebra structure on the $\Bbbk$-module $\CoHAQ$. To distinguish the twisted algebra structure on $\CoHAQ$ without explicitly referring to the multiplication map $\mathsf{m}_\tau$, we add a subscript to the $\Bbbk$-module itself: $\CoHAQ_\tau$.

\begin{defn}
We call the associative unital $\N^I\times \N^I$-graded $\Bbbk$-algebra $(\CoHAQ_\tau, \mathsf{m}_\tau)$ the twisted framed doubled cohomological Hall algebra associated with the quiver $(Q, I)$.
\end{defn}
\begin{remq}
The introduction of the twist is not new in the literature. Indeed, Yang and Zhao introduced in \cite{yang2018CoHA} the same twist to define a canonical morphism between the CoHA associated with the preprojective path algebra $\Ci \overline Q/J_{\mu}$ and the twisted version of Kontsevich-Soibelman CoHA for the doubled quiver $\overline Q$. Here $J_{\mu}$ is the ideal associated with the moment map of the action of $G_\vi$ on $T^*\Rep(\vi)$. Equivalently, via dimensional reduction \cite[Appendix A]{Davison_critical_CoHA}, the twist in the multiplication map $m_\tau$ can be also detected in the critical CoHA for the tripled quiver with a non-trivial potential \cite{ginzburg2007calabiyau, ren2016cohomological, Yang_2019}.

We also remark that, following Yang-Zhao's construction, one could define a framed version of their preprojective cohomological Hall algebra together with an algebra morphism to $\CoHAQ_\tau$. 
\end{remq}

\subsection{Shuffle description for the product}

The framed CoHA admits a very explicit description under the identifications
\[
\CoHA(\vi,\w)=\Bbbk[\liea_\w\times \s_{\vi}]^{W_{G_{\vi}}}.
\]
An element $f\in \CoHA(\vi,\w)$ is simply a polynomial $f(a,s)$, symmetric in each set of Chern roots associated to the groups $\GL(\vi_i)\subset G_{\vi}$.

Fix decompositions $\vi=\vi_1+\vi_2$ and $\w=\w_1+\w_2$ and view both $\CoHA(\vi,\w)$ and $\CoHA(\vi_1,\w_1)\otimes_{\Bbbk}\CoHA(\vi_2,\w_2)$  as subalgebras of 
\[
\Bbbk[a_\w\times \s_{\vi}]=\Bbbk[\liea_{\w_1}\times \s_{\vi_1}]\otimes_\Bbbk\Bbbk[\liea_{\w_2}\times \s_{\vi_2}]
\]
Let 
\[
\Shuffle: \Bbbk[\liea_{\w_1}\times \s_{\vi_1}]^{W_{G_{\vi_1}}}\otimes_\Bbbk\Bbbk[\liea_{\w_2}\times \s_{\vi_2}]^{W_{G_{\vi_2}}}\to \Bbbk[\liea_\w\times \s_{\vi}]^{W_{G_{\vi}}}
\]
be the symmetrization map by shuffles, and recall the definitions of the classes in Section \ref{Subsubsection tangent and normal classes}.
\begin{thm}
\label{CoHA multiplication formula}
The product $\mathsf{m}(f_1,f_2)$ of two elements $f_1(a_1,s_1)\in \CoHA(\vi_1,\w_1)$ and  $f_2(a_2,s_2)\in \CoHA(\vi_2,\w_2)$ is given by 
\[
\Shuffle\left( \frac{e(T^*\Rep(\vi,\w)[-1]) f_1f_2}{e(\g_{\vi}[-1])}   \right).
\]
Similarly, the twisted product $\mathsf{m}_\tau(f_1,f_2)$ is given by 
\[
\Shuffle\left( \frac{e(T^*\Rep(\vi,\w)[-1]) e({\hbar}\g_{\vi}[1])f_1f_2}{e(\g_{\vi}[-1])} \right)
\]
\end{thm}

Notice that the kernel $e(T^*\Rep(\vi,\w)[-1])/e(\g_{\vi}[-1])$ of the multiplication map $\mathsf{m}$ coincides with $e(\mathfrak{N}[-1])$, the Euler class of the negative part of the normal bundle defined in Section \ref{Subsubsection tangent and normal classes}. 

\begin{proof}
We decompose the correspondence \eqref{fundamental correspondence stacks} giving rise to the multiplication map $\mathsf{m}$ as follows
\[
 \begin{tikzcd}
    &  \left[Z/G_{\vi_1}\times G_{\vi_2}\right]\arrow[d, "p_2"]\arrow[dl, "p_1"] &\\
        \RG(\vi_1,\w_1)\times\RG(\vi_2,\w_2) &   \left[Z/P\right]\arrow[d, hook, "i"]\arrow[l, ,swap, "\mathsf{p}"] \arrow[r, "\mathsf{q}"] & \RG(\vi,\w)\\
        & \left[T^*\Rep(\vi,\w)/P\right]\arrow[ur, "\pi"] &
    \end{tikzcd}
\]
All the maps are the obvious ones. In particular, $p_1$ is induced by the projection $Z\to T^*\Rep(\vi_1,\w_1)\times T^*\Rep(\vi_2,\w_2) $ and $p_2$ is an affine bundle, so $\mathsf{p}^*=(p_2^*)^{-1}\circ p_1^*$ is just the identity map in the tautological presentation. The map $i$ is induced by the inclusion $Z\hookrightarrow T^*\Rep(\vi,\w)$ so $i_*$ is multiplication by $e(T^*\Rep(\vi,\w)[-1])$, the Euler class of its normal bundle. Finally, the map $\pi_*$ is modelled by the map 
\[
\tilde\pi:T^*\Rep(\vi,\w)_P\to T^*\Rep(\vi,\w)_{G_{\vi}},
\]
which is a fibration by partial flags $G_{\vi}/P$, and hence $\text{Ker}(d\tilde\pi)=\g_\vi[-1]$. Thus
\begin{equation}
    \label{explcit formula pushforward shuffle}
    \pi_*(f)=\tilde\pi_*(f)=\Shuffle \left(\frac{f}{e(\g_\vi[-1])}\right).
\end{equation}
Altogether,
\[
\mathsf{m}(f_1,f_2)=\mathsf{q}_*\circ \mathsf{p}^*(f_1\otimes f_2)=\pi_*\circ i_*\circ (p_2^*)^{-1}\circ p_1^*(f_1\otimes f_2)=\Shuffle \left(\frac{e(T^*\Rep(\vi,\w)[-1])f_1f_2}{e(\g_\vi[-1])}\right),
\]
as claimed. 

Finally, since $\mathsf{m}_\tau$ is defined as the composition of $\mathsf{m}$ with cup product by $e({\hbar}\g_{\vi}[1])$, also its formula follows. 
\end{proof}

\begin{remq}
    In \cite[Section 2.3]{yang2018sheafify}, Yang and Zhao define a framed shuffle algebra similar to the one introduced in Theorem \ref{CoHA multiplication formula}. The only significant difference from our construction is that in their work the torus $A_\w$ is replaced with $G_{\w}\supseteq A_{\w}$. Consequently, in \cite{yang2018sheafify} an additional shuffle symmetrizing the equivariant parameters $a$ associated with $A_\w$ is considered. 
    
    Almost at the same time, Konno introduced an elliptic (and dynamical) version of the shuffle product from Theorem \ref{CoHA multiplication formula} in the case of type-A quivers with one framing vertex \cite{konno2017elliptic}. This structure was then extended in \cite{botta2021shuffle} to arbitrary quivers and framings.
\end{remq}

\subsection{CoHA via abelian quotients}

\label{subsection abelianization of CoHA}

\subsubsection{}

\label{subsubsection 1 abelianization of CoHA}]
Choose a maximal torus $S\subset G_\vi$ with Lie algebra $\s$.  In this section, we introduce an abelianized version of the CoHA using $S$-equivariant rather than $G_\vi$-equivariant cohomology.

Let $B$ be a Borel subgroup containing $S$ with Lie algebra
$\bi=\s\oplus \n$. We have $\g_{\vi}=\s_{\vi}\oplus \n\oplus \n^{\vee}$ and hence $\s^\perp=\n\oplus\n^\vee$. Consider the following diagram
\begin{equation}
\label{abelianization diagram CoHA}
    \begin{tikzcd}
    \left[T^*\Rep(\vi,\w)/S\right]\arrow[d, "Q"]\arrow[r, "i_+"] & \left[T^*\Rep(\vi,\w)\oplus {\hbar} \n/S\right]\arrow[r, "i_-"] & \left[T^*\Rep(\vi,\w)\oplus {\hbar}\s^\perp/S\right]\\
    \left[T^*\Rep(\vi,\w)/B\right]\arrow[d, "q"]& & \\
    \left[T^*\Rep(\vi,\w)/G_{\vi}\right] & & 
    \end{tikzcd}
\end{equation}
Here, $Q$ is an affine bundle with fiber $B/S$, and $q$ is a fibration by flags $G_{\vi}/B$. As usual, the label ${\hbar}$ indicates the presence of an action of $\Ci_{\hbar}^\times$ with weight one.

All the maps are clearly $T_\w$-equivariant.
Passing to cohomology, notice that 
\begin{itemize}
    \item The map $q$ is proper, and the pushforward $q_*$ is surjective\footnote{Surjectivity can be deduced from the explicit presentation of $q_\ast$. In analogy with \eqref{explcit formula pushforward shuffle}, we have $q_\ast(f)=\text{Sym}\left(\frac{f}{e(\n^\vee)}\right)$ for any $f\in \Bbbk[a_\w\times \s_{\vi}]\cong H_{T_\w}([T^*\Rep(\vi,\w)/B]) $. As a consequence, given any symmetric function $g\in \Bbbk[a_\w\times \s_{\vi}]^{W_{G_{\vi}}}\cong H_{T_\w}([T^*\Rep(\vi,\w)/G])$, the function $f=\frac{1}{|W_{G_{\vi}}|}e(\n^{\vee})g\in \Bbbk[a_\w\times \s_{\vi}]$ satisfies $q_\ast(f)=g$.}.
    \item The pullback $Q^*$ is an isomorphism.
\end{itemize}

From now on, we always assume that the maximal tori $S$ and the Borel subgroups $B$ are chosen compatibly, in the sense that if $S_1$, $S_2$, $B_1$ and $B_2$ are the chosen maximal tori and Borel subgroups of $G_{\vi_1}$ and $G_{\vi_2}$ then $S_1\times S_2\subset G_{\vi_1}\times G_{\vi_2}$ is the chosen maximal torus of $G_\vi\supset G_{\vi_1}\times G_{\vi_2}$ and, similarly, $B_1\times B_2$ is the subgroup of $B$ such that 
\[
\bi= \bi_1\oplus \bi_2\oplus \g_{\vi}[1],
\]
or, equivalently, such that 
\[
\n=\n_1\oplus \n_2\oplus \g_{\vi}[1].
\]
We now introduce the $\Bbbk$-module
\[
\CoHA_{\text{ab},\tau}=\bigoplus_{\vi,\w\in \N^I} \CoHA_{\text{ab},\tau}(\vi,\w)
\]
by setting 
\[
\CoHA_{\text{ab},\tau}(\vi,\w)=H_{T_\w}([T^*\Rep(\vi,\w)\oplus {\hbar}\s^\perp/S])
\]
The goal of this section is to promote this module to a $\N^I\times \N^I$-graded algebra and to recover the twisted CoHA $\CoHA_\tau$ from $\CoHA_{\text{ab},\tau}$. For similar constructions in a different framework, see \cite{Davison_critical_CoHA, padurariu_KHA}.  We remark that the need to add the factor $\s^\perp$ and to introduce the horizontal maps in the diagram \eqref{abelianization diagram CoHA} is because we want to reproduce $\CoHA_\tau$ rather than the untwisted $\CoHA$.

\subsubsection{}

Let $\vi=\vi_1+\vi_2$ and $\w=\w_1+\w_2$. The analog of the correspondence \eqref{fundamental correspondence prequotient} now is 
\begin{equation*}
    \begin{tikzcd}
    (T^*\Rep(\vi_1,\w_1)\oplus {\hbar}\s_1^\perp)\times( T^*\Rep(\vi_2,\w_2)\oplus  {\hbar}\s_2^\perp) &  Z_S \arrow[l] \arrow[r, hook]  & T^*\Rep(\vi,\w)\oplus {\hbar}\s^\perp
    \end{tikzcd}
\end{equation*}
where
\[
Z_S=Z\oplus {\hbar}\s_1^\perp \oplus {\hbar}\s_2^\perp.
\]
Quotienting by $S=S_1\times S_2$, we get 
\begin{equation*}
    \begin{tikzcd}[column sep=2em]
    \left[T^*\Rep(\vi_1,\w_1)\oplus  {\hbar}\s_1^\perp/S_1\right]\times\left[ T^*\Rep(\vi_2,\w_2)\oplus  {\hbar}\s_2^\perp/S_2\right] &  \left[Z_S/S\right] \arrow[l, swap, "\mathsf{p}_S"] \arrow[r, "\mathsf{q}_S"] & \left[T^*\Rep(\vi,\w)\oplus  {\hbar}\s^\perp/S\right].
    \end{tikzcd}
\end{equation*}

We now define multiplication maps
\begin{equation}
    \label{abelian multiplication map}
    \mathsf{m}_{\text{ab},\tau}: \CoHA_{\text{ab},\tau}(\vi_1,\w_1)\otimes_\Bbbk \CoHA_{\text{ab},\tau}(\vi_2,\w_2)\to \CoHA_{\text{ab},\tau}(\vi,\w)
\end{equation}
as the composition of the following morphisms 
\begin{enumerate}
    \item the K\"unneth isomorphism
    \[
    \CoHA_{\text{ab},\tau}(\vi_1,\w_1)\otimes_\Bbbk \CoHA_{\text{ab},\tau}(\vi_2,\w_2)\cong H_{T_\w}^*([T^*\Rep(\vi_1,\w_1)\oplus {\hbar}\s_1^\perp/S_1]\times [T^*\Rep(\vi_2,\w_2)\oplus {\hbar}\s_2^\perp/S_2]),
    \]
    \item the pullback
    \[
    \mathsf{p}_S^*: H_{T_\w}^*([T^*\Rep(\vi_1,\w_1)\oplus {\hbar}\s_1^\perp/S_1]\times [T^*\Rep(\vi_2,\w_2)\oplus {\hbar}\s_2^\perp/S_2])\to H_{T_\w}([Z_S/S]),
    \]
    \item the pushforward
    \[
    (\mathsf{q}_S)_*:  H_{T_\w}([Z_S/S])\to H_{T_\w}([T^*\Rep(\vi,\w)\oplus {\hbar}\s^\perp/S])=\CoHA_{\text{ab},\tau}(\vi,\w).
    \]
\end{enumerate}
An argument similar to that of Theorem \ref{CoHA multiplication formula}, but simpler, shows that the multiplication map \eqref{abelian multiplication map} is simply given by cup product by the class 
\begin{equation}
    \label{explicit formula multiplication abelianized CoHA}
    e(T^*\Rep(\vi,\w)[-1])e({\hbar}\g_\vi[-1])e({\hbar}\g_\vi[1]).
\end{equation}
Similarly to Theorem \ref{theorem CoHA is an algebra}, we also have: 
\begin{propn}
    The maps $\mathsf{m}_{\text{ab},\tau}$ define an $\N^I\times \N^I$-graded associative unital $\Bbbk$-algebra structure on $\CoHA_{\text{ab},\tau}$ with unit element given by $1\in \Bbbk=\CoHA(0,0)$.
    \end{propn}

\subsubsection{}

We now relate $\CoHA_\tau$ with $\CoHA_{\text{ab},\tau}$. 
For simplicity, let us denote 
\[
\RG_{\bi}:=\left[T^*\Rep(\vi,\w)\oplus {\hbar}\n/S\right].
\]
\begin{thm}
\label{theorem abelianization CoHA}
The multiplication map $\mathsf{m}_{\text{ab},\tau}$ uniquely factors through the following commutative diagram
\[
\begin{tikzcd}
 H_{T_{\w_1}}^*(\RG_{\bi_1})\otimes_\Bbbk H_{T_{\w_2}}^*(\RG_{\bi_2}) \arrow[d, dashrightarrow]\arrow[rr, "(i_{1,-})_*\otimes(i_{2,-})_*"]& & \CoHA_{\text{ab},\tau}(\vi_1,\w_1)\otimes_\Bbbk \CoHA_{\text{ab},\tau}(\vi_2,\w_2)\arrow[d, "\mathsf{m}_{\text{ab},\tau}"]\\
H_{T_\w}(\RG_{\bi})\arrow[rr, "(i_-)_*"] & & \CoHA_{\text{ab},\tau}(\vi,\w)
\end{tikzcd}
\]
Denote the dashed arrow by $\mathsf{m}_{\text{ab},\tau}^{\bi}$. Then the following diagram
\[
\begin{tikzcd}
\CoHA_{\tau}(\vi_1,\w_1)\otimes_\Bbbk \CoHA_{\tau}(\vi_2,\w_2)\arrow[d, "\mathsf{m}_{\tau}"]  & & & & & H_{T_{\w_1}}^*(\RG_{\bi_1})\otimes_\Bbbk H_{T_{\w_2}}^*(\RG_{\bi_2})\arrow[d, "\mathsf{m}_{\text{ab},\tau}^{\bi}"]\arrow[lllll, swap,  "(q_{1}\times q_{2})_*\circ ((Q_1\times Q_2)^*)^{-1}\circ (i_{1,+}\times  i_{2,+})^* "] \\
\CoHA_{\tau}(\vi,\w) & & & & &  H_{T_\w}(\RG_{\bi}) \arrow[lllll, swap,  "q_*\circ (Q^*)^{-1}\circ (i_+)^* "]
\end{tikzcd}
\]
is also commutative.
\end{thm}
\begin{proof}
We claim that the dashed arrow of the first diagram is given by the correspondence
\begin{equation*}
    \begin{tikzcd}
    \left[T^*\Rep(\vi_1,\w_1)\oplus  {\hbar}\n_1/S_1\right]\times\left[ T^*\Rep(\vi_2,\w_2)\oplus  {\hbar}\n_2/S_2\right] &  \left[Z_\bi/S\right] \arrow[l, swap, "\mathsf{p}_\bi"] \arrow[r, "\mathsf{q}_\bi"] & \left[T^*\Rep(\vi,\w)\oplus {\hbar}\n/S\right]
    \end{tikzcd}
\end{equation*}
with 
\[
Z_\bi= Z\oplus {\hbar}\n_1\oplus {\hbar}\n_2.
\]
It fits in the following pair of Cartesian squares
\begin{equation*}
    \begin{tikzcd}[column sep=2em]
    \left[T^*\Rep(\vi_1,\w_1)\oplus  {\hbar}\n_1/S_1\right]\times\left[ T^*\Rep(\vi_2,\w_2)\oplus  {\hbar}\n_2/S_2\right]\arrow[d, "(i_{1,-})\times (i_{2,-})"] &  \left[Z_\bi/S\right]\arrow[d, "i"] \arrow[l, swap, "\mathsf{p}_\bi"] \arrow[r, "\mathsf{q}_\bi"] & \left[T^*\Rep(\vi,\w)\oplus  {\hbar}\n/S\right]\arrow[d, "i_-"]\\
    \left[T^*\Rep(\vi_1,\w_1)\oplus  {\hbar}\s_1^\perp/S_1\right]\times\left[ T^*\Rep(\vi_2,\w_2)\oplus  {\hbar}\s_2^\perp/S_2\right] &  \left[Z_S/S\right] \arrow[l, swap, "\mathsf{p}_S"] \arrow[r, "\mathsf{q}_S"] & \left[T^*\Rep(\vi,\w)\oplus  {\hbar}\s^\perp/S\right]
    \end{tikzcd}
\end{equation*}
where the central arrow $i$ is induced by the inclusion map 
\[
Z_\bi= Z\oplus {\hbar}\n_1\oplus {\hbar}\n_2\hookrightarrow Z\oplus {\hbar}\s_1^\perp\oplus \hbar\s_2^\perp
\]
By compatibility of pullback and pushforward on Cartesian squares, we deduce that
\begin{align*}
    (\mathsf{q}_S)_*\circ \mathsf{p}_S^*\circ ((i_{1,-})_*\otimes(i_{2,-})_*)&=(\mathsf{q}_S)_*\circ i_*\circ \mathsf{p}_\bi^*\\
    &=(i_-)_*\circ (\mathsf{q}_\bi)_*\circ \mathsf{p}_\bi^*,
\end{align*}
as claimed. Notice that this implies that the map $\mathsf{m}_{\text{ab},\tau}^{\bi}$ is the cup product operation by 
\[
e(T^*\Rep(\vi,\w)[-1])e({\hbar}\g_\vi[1]).
\] 
We now prove the commutativity of the second diagram. Firstly, notice that the pullbacks $(i_+)^*$ and $(i_{1,+})^*\otimes (i_{2,+})^*$ are isomorphisms, and the composition $(i_+)^*\circ \mathsf{m}_{\text{ab},\tau}^{\bi}\circ ((i_{1,+})^*\otimes (i_{2,+})^*)^{-1}$ is given by $(\tilde{\mathsf{q}}_S)_*\circ (\tilde{\mathsf{p}}_S)^*\circ e({\hbar}\g_{\vi}[1])$, where $\tilde{\mathsf{p}}_S$ and $\tilde{\mathsf{q}}_S$ are maps fitting in the correspondence 
\begin{equation*}
    \begin{tikzcd}
    \left[T^*\Rep(\vi_1,\w_1)/S_1\right]\times\left[ T^*\Rep(\vi_2,\w_2)/S_2\right] &  \left[Z/S\right] \arrow[l, swap, "\tilde{\mathsf{p}}_S"] \arrow[r, "\tilde{\mathsf{q}}_S"] & \left[T^*\Rep(\vi,\w) /S\right].
    \end{tikzcd}
\end{equation*}
Since $\mathsf{m}_{\tau}$ is by definition $\mathsf{m}\circ e({\hbar}\g_{\vi}[1])$, to complete the proof, it suffices to show that the following diagram
\begin{equation}
    \label{diagram proof abelianization coha}
    \begin{tikzcd}
\CoHA_{\tau}(\vi_1,\w_1)\otimes_\Bbbk \CoHA_{\tau}(\vi_2,\w_2)\arrow[d, "\mathsf{m}=\mathsf{q}_*\circ \mathsf{p}^*"]  & & & H_{T_{\w_1}}^*(\left[T^*\Rep(\vi_1,\w_1)/S_1\right])\otimes_\Bbbk H_{T_{\w_2}}^*(\left[T^*\Rep(\vi_2,\w_2)/S_2\right])\arrow[d, "(\tilde{\mathsf{q}}_S)_*\circ (\tilde{\mathsf{p}}_S)^*"]\arrow[lll, swap,  "(q_{1}\times q_{2})_*\circ ((Q_1\times Q_2)^*)^{-1} "] \\
\CoHA_{\tau}(\vi,\w) & & &  H_{T_\w}(\left[T^*\Rep(\vi,\w)/S\right]) \arrow[lll, swap,  "q_*\circ (Q^*)^{-1} "]
\end{tikzcd}
\end{equation}
is commutative. Consider the commutative diagram
\[
\begin{tikzcd}
\left[T^*\Rep(\vi_1,\w_1)/S_1\right]\times\left[ T^*\Rep(\vi_2,\w_2)/S_2\right]\arrow[d, "Q_1\times Q_2"] &  \left[Z/S\right]\arrow[d] \arrow[l, swap, "\tilde{\mathsf{p}}_S"] \arrow[r, "\tilde{\mathsf{q}}_S"] & \left[T^*\Rep(\vi,\w) /S\right]\arrow[d, "Q"]\\
\left[T^*\Rep(\vi_1,\w_1)/B_1\right]\times \left[T^*\Rep(\vi_2,\w_2)/B_2\right]\arrow[d, "q_1\times q_2"] & \left[Z/B\right]\arrow[d]\arrow[l]\arrow[r] & \left[T^*\Rep(\vi,\w)/B\right]\arrow[d, "q"]\\
\left[T^*\Rep(\vi_1,\w_1)/G_{\vi_1}\right]\times \left[T^*\Rep(\vi_2,\w_2)/G_{\vi_2}\right] &   \left[Z/P\right]\arrow[l, ,swap, "\mathsf{p}"] \arrow[r, "\mathsf{q}"] & \left[T^*\Rep(\vi,\w) /G_{\vi}\right]
\end{tikzcd}
\]
where the unlabelled maps are the obvious ones. The squares in the lower left and upper right are Cartesian, so the commutativity of diagram \eqref{diagram proof abelianization coha} follows from the compatibility of pushforward and pullback on Cartesian squares. Details are left to the reader. 
\end{proof}

\section{Cohomological stable envelopes}

\subsection{Chambers, attracting sets and polarization}

Let $X$ be a smooth quasi-projective symplectic variety with an action a torus $T$ rescaling the symplectic form with weight ${\hbar}^{-1}$ and let $A\subset \ker({\hbar})$. In this work, we will be interested in the case when
\[
X=T^*\Rep(\vi,\w)////^\theta H
\]
where $H$ is either $G_\vi$ or a maximal torus $S\subset G_\vi$ thereof.

\subsubsection{}

\begin{defn}[\cite{maulik2012quantum}]
A chamber $\mathfrak{C}$ is a connected component of $(\text{Cochar}(A)\otimes_\Z \R)\setminus \Delta$, where $\Delta$ is the hyperplane arrangement determined by the $A$-weights of the normal bundles of $X^A$. 
One says that a weight $\chi\in \Char(A)$ is attracting (resp. repelling) with respect to $\mathfrak{C}$ if $\lim_{t\to 0} \chi\circ \sigma (t)=0$ (resp. if $\lim_{t\to 0} \chi\circ \sigma (t)=\infty$) for any $\sigma\in \mathfrak{C}$. The trivial character is called the $A$-fixed character.
\end{defn} 

Restricted to some fixed component $F\subset X^A$, any $A$-equivariant $K$-theoretic class $V$ of $X$ decomposes according to the characters of $A$ in attracting, $A$-fixed and repelling classes
\[
\restr{V}{F}=\restr{V}{F,+}+\restr{V}{F,0}+\restr{V}{F,-}
\]
with respect to a chosen chamber $\mathfrak{C}$. For instance, we have that 
\[
\restr{TX}{F}=\restr{TX}{F,+}+TF+\restr{TX}{F,-},
\]
where $N_F^+=\restr{TX}{F,+}$ and $N_F^-=\restr{TX}{F,-}$ are the attracting and repelling parts of the normal bundle $N_F$ to $F$ in $X$.

If $X$ is a Nakajima variety and $A=A_\w$, these geometrically defined chambers coincide with the Lie-theoretic chambers.
Indeed, the $A_\w$-weights of the normal bundles of the $A_\w$-fixed components of a Nakajima variety $X$ are the roots of the group $\GL(\sum_j \w_j)$, and hence the hyperplane arrangement $\Delta$ coincides with the one that defines the Lie-theoretic chambers of a maximal torus of $\GL(\sum_j \w_j)$.

\subsubsection{}

Let $F$ be a fixed component in $X^{A}$. The attracting set 
\[
\Att{C}{F}=\Set{x\in X}{\lim_{t\to0} \sigma(t)\cdot x\in F, \text{ for some $\sigma \in \mathfrak{C}$}}
\]
is well defined, i.e. independent of the choice of $\sigma\in \mathfrak{C}$, and it is also an affine bundle over $F$. This follows from the standard results in \cite{BBdecomposition}. Notice in particular that $\restr{T(\Att{C}{F})}{F}=N^+_F$.

We also recall the following key definition:
\begin{defn}
The full attracting set $\text{Att}^f_{\mathfrak{C}}(F)$ is the minimal closed subset of X containing $F$ and closed under taking $\Att{C}{-}$.
\end{defn}
\subsubsection{}

The choice of a chamber determines a partial ordering on the set of fixed components of $X^A$ by taking the transitive closure of the relation
\[
\Att{C}{F}\cap F'\neq \emptyset \quad \Rightarrow \quad F\geq F'.
\]
See \cite[Section 3.2.3]{maulik2012quantum} for a proof of the fact that this is really a partial order. 
\subsection{Stable envelopes}

We now recall Maulik-Okounkov's definition of stable envelopes.



\subsubsection{}

Since $A$ acts trivially on $X^A$, we have
\[
H_T(X^A)\cong H_{T/A}(X^A)\otimes H_A(\pt),
\]
where $H_A(\pt)$ is the ring of polynomials in equivariant parameters in $A$. Although the isomorphism above is not canonical, different isomorphisms give the same filtration of $H_T(X^A)$ by the degree $\deg_A$ in the equivariant parameters. 

With this observation, we can recall the definition of cohomological stable envelope.
\begin{thm}[\cite{maulik2012quantum}, Theorem 3.3.4]
Let $\mathfrak{C}$ be a chamber for the action of $A$ on $X$. There exist a unique map of $H_{T}(\pt)$-modules 
\[
\text{Stab}_{\mathfrak{C}}: H_T(X^A)\to H_T(X)
\]
such that
\begin{enumerate}
    \item For every fixed component $F\subset X^A$, then 
    \[
    \restr{\Stab{}{C}{F}}{F}= e(N^-_F),
    \]
    where $\Stab{}{C}{F}$ is the restriction of $\text{Stab}_{\mathfrak{C}}$ to $H_{T}(F)\subset H_T(X^A) $.
    \item $\Stab{}{C}{F}$ is supported on $\text{Att}^f_{\mathfrak{C}}(F)$, i.e. 
    \[
    \restr{\Stab{}{C}{F}}{U}=0,
    \]
    where $U=X\setminus \text{Att}^f_{\mathfrak{C}}(F)$.
    \item $\deg_A(\restr{\Stab{}{C}{F}(\gamma)}{F'})< \frac{1}{2}\text{codim}_{X}(F')$ for all $\gamma\in H_{T/A}(F)$ and $F'< F$.
\end{enumerate}
\end{thm}

\subsubsection{}

The next lemma is a key result for the application of stable envelopes in the geometric representation theory of Yangians, see \cite{maulik2012quantum}, and it also lies at the heart of the interpretation of stable envelopes in terms of CoHAs that we shall give.
\begin{lma}[{\cite[Section 3.6]{maulik2012quantum}}]
\label{triangle lemma}
Let $\mathfrak{C}$ be a chamber for the action of $A$ on $X$ and let $\mathfrak{C}'\subset \mathfrak{C}$ be a face of some dimension. Let $A'$ be torus whose Lie algebra is the span of $\mathfrak{C}'$ in the Lie algebra of $A$ and let $\mathfrak{C}/\mathfrak{C}'$ be the projection of $\mathfrak{C}$ on the Lie algebra of $A/A'$.
Then the following diagram
\begin{equation*}
    \begin{tikzcd}
    H_T(X^A)\arrow[dr, swap,  "\text{Stab}_{\mathfrak{C}/\mathfrak{C}'}"]\arrow[rr, "\text{Stab}_{\mathfrak{C}}"] & & H_T(X)\\
    & H_T(X^{A'})\arrow[ur, swap, "\text{Stab}_{\mathfrak{C}'}"] &
    \end{tikzcd}
\end{equation*}
is commutative.
\end{lma}

\subsection{Stable envelopes of Nakajima varieties: R-matrices}
\label{subsection Stable envelopes of Nakajima varieties: combinatorics}
\subsubsection{}
Let now $X$ be a Nakajima variety $X=\naka(\vi,\w)$ equipped with the action of the torus $T=T_\w$. A decomposition $\w=\w_1+\w_2\dots+\w_k$ gives a homomorphism 
\[
A=\lbrace (a_1,a_2,\dots ,a_k) \mid a_i\in \Ci^\times  \rbrace\hookrightarrow A_\w.
\]
By iteration of Proposition \ref{propoistion fixed point naka}, we get
\begin{equation}
    \label{tensor decomposition nakajima varieties}
    H_{T_\w}(\naka(\vi,\w)^A)=\bigoplus_{\sum_{j=1}^k \vi_j=\vi} \bigotimes_{j=1}^k H_{T_{\w_j}}(\naka(\vi_j,\w_j)).
\end{equation}
and hence the stable envelope is a collection of maps 
\[
H_{T_{\w_1}}(\naka(\vi_1,\w_1))\otimes_\Bbbk \dots\otimes_\Bbbk H_{T_{\w_k}}(\naka(\vi_k,\w_k)) \to H_{T_\w}(\naka(\vi,\w)).
\]

\subsubsection{}

For this torus action, there are $k!$ possible chambers, namely
\[
\mathfrak{C}_{\sigma}=\lbrace a_{\sigma(1)}<a_{\sigma(2)}<\dots  <a_{\sigma(k)}\rbrace \qquad \sigma \in S_k.
\]
The change of basis matrix relating stable envelopes for two different chambers $\mathfrak{C}_{\sigma}$ and $\mathfrak{C}_{\tau}$, namely the map
\[
R_{\mathfrak{C}_{\tau}, \mathfrak{C}_{\sigma}}:=\text{Stab}_{\mathfrak{C}_{\tau}}^{-1}\circ \text{Stab}_{\mathfrak{C}_{\sigma}}\in H_{T_\w}(X^A)\otimes \Ci(\text{Lie}(A))
\]
is called $R$-matrix and is a fundamental object in the geometric representation theory of Yangians developed by Maulik and Okounkov \cite{maulik2012quantum}. In particular, when $k=3$ different factorizations of $R_{\mathfrak{C}_{\tau}, \mathfrak{C}_{\sigma}}$ provide solutions of the classical Yang-Baxter equation with spectral parameters.

\subsection{Stable envelopes of hypertoric varieties: explicit formulas}

\subsubsection{}

If we consider the GIT-symplectic reduction of $T^*\Rep(\vi,\w)$ by a maximal torus $S\subset G_\vi$ with stability condition $\theta\in\Char(S)$
\[
\naka_S(\vi,\w):=T^*\Rep(\vi,\w)////^\theta S=\mu_S^{-1}(0)^{\theta-ss}/S,
\]
we produce a hypertoric variety rather than a Nakajima quiver variety. Like quiver varieties, they come equipped with an action of $T_\w$ and, for a generic stability condition, they are smooth symplectic varieties, so it makes sense to talk about their stable envelopes. Moreover, also for hypertoric varieties, Kirwan surjectivity holds \cite{Proudfoot_hyper}, and hence 
\[
H_{T_\w}(\naka_S)\cong \Bbbk[\liea_\w\times \s]/I,
\]
where $I$ is some ideal. For an explicit description of the generators of $I$, see \cite{Proudfoot_hyper}.

Remarkably, stable envelopes of hypertoric varieties admit a very explicit description in terms of tautological classes. This description was originally developed in full generality by Shenfeld \cite{Shenfeld}, but for our interests, it suffices to have an explicit formula for the stable envelope map of the form
\begin{equation*}
    \text{Stab}_{\lbrace a< 0\rbrace }: H_{T_\w}(\naka_{S_1}(\vi_1,\w_1)\times \naka_{S_2}(\vi_2,\w_2))\to H_{T_\w}(\naka_S(\vi_1+\vi_2,\w_1+\w_2))
\end{equation*}
Here, $S=S_1\times S_2$ and the torus action considered is the one of the one-dimensional torus $\Ci^\times\subset A_\w$ acting with weight one on $W_1$ and trivially elsewhere. For this one-dimensional torus action, the chamber arrangement is particularly easy. Indeed, we have $\text{Cochar}(\Ci^\times)\otimes_\Z\R=\R$, so are only two chambers: the positive axis $\lbrace a>0 \rbrace$ and the negative axis $\lbrace a<0 \rbrace$ .

Notice that the cohomology 
\begin{equation}
    \label{cohomology fixed component hypertoric variety}
    H_{T_\w}(\naka_{S_1}(\vi_1,\w_1)\times \naka_{S_2}(\vi_2,\w_2))
\end{equation}
of the fixed component $\naka_{S_1}(\vi_1,\w_1)\times \naka_{S_2}(\vi_2,\w_2)\subset \naka_{S}(\vi_1+\vi_2,\w_1+\w_2)^{\Ci^\times}$ is isomorphic to
\begin{equation*}
    \Bbbk[\liea_{\w_1}\times \s_1]/I_1\otimes_\Bbbk \Bbbk[\liea_{\w_2}\times \s_2]/I_2\cong \Bbbk[\liea_\w\times \s]/(I_1+I_2)
\end{equation*}
hence both the domain and the codomain of the stable envelope map $\text{Stab}_{\lbrace a<0\rbrace}$ are quotients of the same ring $\Bbbk[\liea_\w\times \s]$.

\begin{propn}
\label{propn explicit formula stab hypertoric}
Let $p_\gamma\in\Bbbk[\liea_\w\times \s] $ be a polynomial representing a class $\gamma \in H_{T_\w}(\naka_{S_1}(\vi_1,\w_1)\times \naka_{S_2}(\vi_2,\w_2))$. Then the polynomial 
\[
 e(T^*\Rep(\vi,\w)[-1])p_\gamma
\]
represents the class $\text{Stab}_{\lbrace a< 0\rbrace }(\gamma)$.
\end{propn}
For the proof, see \cite[Section 4.3]{Shenfeld}. The fact that this assignment gives a well-defined map can be also seen as follows: by equivariant formality, the $H_{T_\w}(\pt)$-module \eqref{cohomology fixed component hypertoric variety} is free\footnote{Indeed, hypertoric varieties are GKM, which is a stronger condition than equivariant formality, see \cite{Harada_GKM}.}, so one can pick a set of polynomials $\lbrace{p_{\gamma_i}\rbrace}_{i}$ representing a basis $\lbrace{\gamma_i\rbrace}_{i}$ of $H_{T_\w}(\naka_{S_1}(\vi_1,\w_1)\times \naka_{S_2}(\vi_2,\w_2))$ and use the assignment 
\[
\gamma_i\mapsto e(T^*\Rep(\vi,\w)[-1])p_{\gamma_i}
\]
to define a $H_{T_\w}(\pt)$-linear map to $H_{T_\w}(\naka_S(\vi_1+\vi_2,\w_1+\w_2))$. Since this map satisfies the axioms of a stable envelope (see \cite{Shenfeld}), by uniqueness of the latter, it is actually independent of the choice of the basis and of its polynomial representatives.
\subsection{Abelianization of stable envelopes}
\label{abelianization of stable envelopes}


\label{subsection abelianization of stable envelopes}

\subsubsection{}

In this section, we recall the abelianization of stable envelopes, which is a procedure that relates the stable envelopes of a Nakajima variety
\[
\naka(\vi,\w)=T^*\Rep(\vi,\w)////G_\vi=\mu_{\vi,\w}^{-1}(0)^{G_\vi-ss}/G_{\vi}
\]
with the ones of the abelian quotient 
\[
\naka_S(\vi,\w):=T^*\Rep(\vi,\w)////S=\mu_S^{-1}(0)^{S-ss}/S.
\]
In the following, we assume that the stability condition $\theta\in \Char(G_\vi)\subset \Char(S)$, which determines the $(\theta, G_\vi)$-semistable locus $\mu_{\vi,\w}^{-1}(0)^{G_{\vi}-ss}$ and $(\theta, S)$-semistable locus $\mu_{S}^{-1}(0)^{S-ss}$ respectively, is chosen generically, in such a way that both the varieties above are smooth.
Abelianization of stable envelopes was first developed by Shenfeld \cite{Shenfeld} for singular cohomology and then reproduced in elliptic cohomology by Aganagic and Okounkov in \cite{aganagic2016elliptic}. The following version is the cohomological limit of abelianization in elliptic cohomology. 
Let $B$ be a Borel subgroup of $G_\vi$ containing $S$, and let $\s\subset \bi \subset \g_{\vi}$ be the corresponding Lie algebras. Since $\mu_S$ is the projection of $\mu_{\vi,\w}$ to $\s^*$, which we identify with $\s$ via trace pairing, we have
\[
\naka_S(\vi,\w)= \mu_{\vi,\w}^{-1}(\s^\perp)^{S-ss}/S,
\]
where $\s^\perp=\n\oplus \n^\vee$ is the subspace of $\g_{\vi}$ vanishing on $\s$.
Consider the diagram
\begin{equation}
\label{abelianization diagram for naka}
    \begin{tikzcd}
    \mu_{\vi,\w}^{-1}(0)^{G_{\vi}-ss}/S \arrow[r, "j_+"]\arrow[d, "\Pi"] &\mu_{\vi,\w}^{-1}(\bi^\perp)^{S-ss}/S \arrow[r, "j_-"]& \mu_{\vi,\w}^{-1}(\s^\perp)^{S-ss}/S  \\
    \mu_{\vi,\w}^{-1}(0)^{G_{\vi}-ss}/B \arrow[d, "\pi"] & &\\
    \mu_{\vi,\w}^{-1}(0)^{G_{\vi}-ss}/G & & 
    \end{tikzcd}
\end{equation}Notice the similarity with the abelianization diagram for the CoHA \eqref{abelianization diagram CoHA}. Indeed, the latter was inspired by diagram \eqref{abelianization diagram for naka}.

Notice that $\Pi$ is an affine bundle with fiber $B/S$, $\pi$ is a fibration by flags $G_{\vi}/B$, and $j_-$ is a closed embedding.

\subsubsection{}
All the maps in the diagram are $T_\w$-equivariant, so the diagram above restricts to fixed components of the action of some torus $A\subset T_\w$. Let $\pi^A$, $\Pi^A$, etc. be the restrictions of the maps above. 
Passing to cohomology, notice that 
\begin{itemize}
    \item The maps $\pi$ and $\pi^A$ are proper, and the pushforwards $\pi_*$ and $\pi^A_*$ are surjective\footnote{Since the map $\pi$ is the restriction of $q: [T^*\Rep(\vi,\w/B]\to [T^*\Rep(\vi,\w/G]$, its surjecticity follows from surjectivity of $q_*$ and Kirwan surjectivity, discussed in sections \ref{subsubsection 1 abelianization of CoHA} and \ref{kirwan surjectivity}, respectively. By Proposition \ref{propoistion fixed point naka}, surjectivity of $\pi$ for arbitrary dimension vectors $(\vi,\w)$ implies surjectivity of $\pi^A_\ast$.}
    \item The pullbacks $\Pi^*$ and $(\Pi^A)^*$ are isomorphisms.
\end{itemize}

Fix some component $F\subset \naka(\vi,\w)^A$ and a chamber $\mathfrak{C}$ for the action of $A$. Let $F_B$ be the (unique) $A$-fixed component in $\pi^{-1}(F)$ whose normal weights in  $\pi^{-1}(F)$ are non-attracting, i.e. 
\[
\restr{\text{Ker}(d\pi)}{F_B, +}=0.
\]
Overall, we get a diagram
\[
\begin{tikzcd}
(\Pi^A)^{-1}(F_B) \arrow[r, "j^A_+"]\arrow[d, "\Pi^A"] &F'_S \arrow[r, "j^A_-"]& F_S \\
F_B \arrow[d, "\pi^A"] & &\\
F & & 
\end{tikzcd}
\]
whose vertices are $A$-fixed components. 
\subsubsection{}
We can now review the statement of the abelianization theorem.
\begin{thm}[{\cite[Section 4.3]{aganagic2016elliptic}}]
\label{theorem abelianization stable envelopes}
The composition $\Stab{}{C}{F_S}\circ (j^A_-)_*$ factors through $(j_-)_*$, i.e. there exists a map making the diagram
\[
\begin{tikzcd}
H_{T_\w}(F'_S)\arrow[r, "(j^A_-)_*"]\arrow[d, dashrightarrow] & H_{T_\w}(F_S)\arrow[d, "\Stab{}{C}{F_S}"]\\
H_{T_\w}(\mu^{-1}(\bi^\perp)^{S-ss}/S)\arrow[r, "(j_-)_*"] & H_{T_\w}(\naka_S(\vi,\w))
\end{tikzcd}
\]
commute.
Denote a choice of dashed arrow by $\Stab{\bi}{C}{F_S}$. Then the following diagram also commutes 
\[
\begin{tikzcd}
H_{T_\w}(F)\arrow[d, "\Stab{}{C}{F}"]  & & & H_{T_\w}(F_S')\arrow[d, "\Stab{\bi}{C}{F_S}"]\arrow[lll, swap,  "(\pi^A)_*\circ ((\Pi^A)^*)^{-1}\circ (j^A_+)^* "] \\
H_{T_\w}(\naka(\vi,\w)) & & &  H_{T_\w}(\mu^{-1}(\bi^\perp)^{S-ss}/S) \arrow[lll, swap,  "\pi_*\circ (\Pi^*)^{-1}\circ (j_+)^* "]
\end{tikzcd}
\]
\end{thm}
Since each of the maps in the composition $(\pi^A)_*\circ ((\Pi^A)^*)^{-1}\circ (j^A_+)^* $ is surjective, the previous theorem expresses the stable envelopes of the Nakajima variety $\naka(\vi,\w)$ in terms of those of the hypertoric variety $\naka_S(\vi,\w)$.

\begin{remq}
Actually, in the original reference \cite{aganagic2016elliptic} Theorem \ref{theorem abelianization stable envelopes} is presented in a slightly different way. Indeed, the composition $(\Pi^*)^{-1}\circ(j_+)^*$ is replaced with pullback by a single map 
\[
\tilde{j}_+ :\mu_{\vi,\w}^{-1}(0)^{G_{\vi}-ss}/B\to \mu_{\vi,\w}^{-1}(\bi^\perp)^{S-ss}/S,
\]
which is the canonical map obtained from the $C^\infty$-isomorphism of $G/B\cong U/U\cap S$ bundles 
\begin{equation}
    \label{isomoprhism flag bundles abelianization}
    \mu_{\vi,\w}^{-1}(0)^{G_{\vi}-ss}/B\cong \mu_{\vi,\w}^{-1}(0)\cap \mu_{\R}^{-1}(\eta)/ (U\cap S)
\end{equation}
and the inclusion 
\[
\mu_{\vi,\w}^{-1}(0)\cap \mu_{\R}^{-1}(\eta)/ (U\cap S)\hookrightarrow \mu_{\vi,\w}^{-1}(\bi^\perp)^{S-ss}/ S.
\]
Here $U$ is a compact form of $G_\vi$ and the real moment map $\mu_{\R}^{-1}(\eta)$ plays the role of the stability condition \cite{Nakajimainstantons}. 
Similarly, also the composition $((\Pi^A)^*)^{-1}\circ(j^A_+)^*$ is replaced with $(\tilde{j}_-^A)^*$.
Despite appearances, it is easy to see that $(\tilde{j}_+)^*=(\Pi^*)^{-1}\circ(j_+)^*$ and $((\Pi^A)^*)^{-1}\circ(j^A_+)^*=(\tilde{j}_-^A)^*$ under the identification \eqref{isomoprhism flag bundles abelianization}. Indeed, the cohomology of $ \mu_{\vi,\w}^{-1}(0)^{G_{\vi}-ss}/B$ is generated by the Chern roots of the tautological bundles, and both $(\tilde{j}_+)^*$ and $(\Pi^*)^{-1}\circ(j_+)^*$ send Chern roots to Chern roots. Hence, they are the same map. This argument proves the first equality; the proof of the second one is similar. Therefore, the statement of the theorem above differs from the original one \cite{aganagic2016elliptic} only in the notation. However, in this article, we avoid using the maps $\tilde{j}_+$ and $\tilde{j}^A_+$ because they are not algebraic. 
\end{remq}


\section{Framed semistable CoHA via stable envelopes}

\subsection{The framed semistable preprojective CoHA}
\label{Section framed semistable CoHA}
\subsubsection{}

Fix a quiver $Q$ and consider the graded $\Bbbk$-module
\[
\HhQ:=\bigoplus_{\vi,\w\in  \N^I} \HhQ(\vi,\w),
\]
where 
\[
\HhQ(\vi,\w):=H_{T_\w}(\naka_Q(\vi,\w)).
\]
As usual, we will drop the reference to the quiver whenever this is clear from the context.

We now define a graded multiplication on $\Hh$ via stable envelopes 
\[
\text{Stab}_{\lbrace a_{1}<a_2 \rbrace}: \Hh(\vi_1,\w_1)\otimes_\Bbbk \Hh(\vi_2,\w_2)\to \Hh(\vi_1+\vi_2,\w_1+\w_2)
\]
\begin{thm}
The previous multiplication map defines an associative unital $\N^I\times \N^I$-graded $\Bbbk$-algebra structure on $\Hh$ with unit given by the element $1\in \Bbbk=\Hh_{0,0}$.
\end{thm}
\begin{proof}
Let $\vi=\vi_1+\vi_2+\vi_3$ and $\w=\w_1+\w_2+\w_3$ and consider the stable envelope map
\[
\text{Stab}_{\lbrace a_1<a_2<a_3 \rbrace}: \Hh(\vi_1,\w_1)\otimes_\Bbbk \Hh(\vi_2,\w_2)\otimes_\Bbbk \Hh(\vi_3,\w_3)\to \Hh(\vi,\w).
\]
Applying Lemma \ref{triangle lemma} twice with $\mathfrak{C}=\lbrace{ a_1<a_2<a_3 \rbrace}$ and $\mathfrak{C}/\mathfrak{C}'$ equal either to $\lbrace{ a_1<a_2 \rbrace}$ or to $\lbrace{ a_2<a_3 \rbrace}$, we get exactly the commutative diagram
\[
\begin{tikzcd}
& \Hh(\vi_1,\w_1)\otimes_\Bbbk\Hh(\vi_2+\vi_3,\w_2+\w_3)\arrow[rd, "\text{Stab}_{\lbrace a_{1}<a_{23} \rbrace}"] & \\
\Hh(\vi_1,\w_1)\otimes_\Bbbk \Hh(\vi_2,\w_2)\otimes_\Bbbk \Hh(\vi_3,\w_3) \arrow[rr, "\text{Stab}_{\lbrace a_1<a_2<a_3 \rbrace}"]\arrow[rd, swap, "\text{Stab}_{\lbrace a_1<a_2 \rbrace}\otimes 1"] \arrow[ru, "1\otimes \text{Stab}_{\lbrace a_2<a_3 \rbrace}"] & & \Hh(\vi,\w) \\
& \Hh(\vi_1+\vi_2,\w_1+\w_2)\otimes_\Bbbk \Hh(\vi_3,\w_3)\arrow[ru, swap, "\text{Stab}_{\lbrace a_{12}<a_3 \rbrace}"] & 
\end{tikzcd}
\]
that proves the claim.
\end{proof}

\begin{defn}
We call the associative unital $\N^I\times \N^I$-graded $\Bbbk$-algebra $(\HhQ, \text{Stab})$ the framed semistable preprojective cohomological Hall algebra of $Q$. 
\end{defn}
\begin{remq}
Unlike in $\CoHA$, some graded components $\Hh_{\vi,\w}$ of $\Hh$ are zero. This is a consequence of the fact that, because of the stability condition, some Nakajima varieties $\naka(\vi,\w)$ are empty. However, the Nakajima variety $\naka(0,0)$ is a singleton for every quiver $Q$ and stability condition, and hence $\Hh(0,0)$ is always equal to $\Bbbk$.
\end{remq}
\begin{remq}
One could naively try to define a CoHA structure on $\HhQ$ by means of a correspondence of the form 
\begin{equation*}
    \begin{tikzcd}
    (\mu_{\vi_1,\w_1}^{-1}(0))^{ss}\times (\mu_{\vi_2,\w_2}^{-1}(0))^{ss} &  (Z\cap \mu_{\vi,\w}^{-1}(0))^{ss}\arrow[l, dashed] \arrow[r, hook]  & (\mu_{\vi,\w}^{-1}(0))^{ss}
    \end{tikzcd}
\end{equation*}
which is the restriction of \eqref{fundamental correspondence prequotient} to the appropriate zero loci and semistable loci. However, it is easy to check that the dashed map is generally not well defined because it does not respect semistability. Substituting $(Z\cap \mu_{\vi,\w}^{-1}(0))^{ss}$ with the open subset for which the dashed map is well defined does not solve the problem because then the inclusion inside $ (\mu_{\vi,\w}^{-1}(0))^{ss}$ is not proper anymore, and hence it does not descend to a pushforward in cohomology. Stable envelopes can be interpreted as a refined correspondence, obtained by iterated corrections on lower strata, that guarantees both the existence of a dashed map and the properness of the solid arrow (the corrections come with multiplicities, so, in general, the variety $(Z\cap \mu_{\vi,\w}^{-1}(0))^{ss}$ is replaced with a Borel-Moore cycle, i.e. a collection of varieties with multiplicities). As a consequence, one can look at $\text{Stab}_{\mathfrak{C}}$ as the natural substitute of the correspondence giving rise to a CoHA structure on the cohomology of Nakajima varieties. 
\end{remq}

\subsection{Framed CoHA vs. framed semistable preprojective CoHA}

\subsubsection{}
In this section, we review Aganagic-Okounkov's construction of a map\footnote{AO denoted this map by $\mathbf{s}$. The letter $``s''$ being overused in this paper, we substitute it with $\bm\psi$.} 
\[
\bm\psi: \Hh(\vi,\w)\to \CoHA(\vi,\w)
\]
that is, up to a multiplicative factor, a section of the restriction map 
\begin{equation*}
    \label{formula non-abelian stab intro section}
    \bm{j}^*:H_{T_\w}(\RG(\vi,\w))\twoheadrightarrow H_{T_\w}(\naka(\vi,\w))
\end{equation*}
associated with the inclusion
\[
\bm{j}:\naka(\vi,\w)= \mu^{-1}_{\vi,\w}(0)^{G_{\vi}-ss}/G_{\vi}\hookrightarrow [T^*\Rep(\vi,\w)/ G_{\vi}] = \RG(\vi,\w).
\]
Like the stable envelopes, the map $\bm\psi$ solves an extension problem, which goes under the name of non-abelian stable envelope \cite{okounkov2020inductiveII}.  One can think of the ordinary ``abelian'' stable envelopes\footnote{We urge the reader to distinguish between the abelianized stable envelopes from Section \ref{abelianization of stable envelopes}, which are the stable envelopes for the hypertoric variety associated with a Nakajima variety, and the ``abelian'' stable envelopes, which are just the ordinary Maulik-Okounkov stable envelopes as in \eqref{stab in intro nonabelian stab}.}
\begin{equation}
\label{stab in intro nonabelian stab}
    \Stab{}{C}{F}: H_{T_\w}(F)\to  H_{T_\w}(\naka(\vi,\w))
\end{equation}
as a map extending a topological class defined on some attracting locus $\Att{C}{F}\subset \naka(\vi,\w)$ to a class on $\naka(\vi,\w)$ supported on $\text{Att}_{\mathfrak C}^{f}(F)$. This extension is obtained by iteration, progressively adding attracting sets sprouting up from the boundary. Analogously, the non-abelian stable envelope $\bm\psi$ extends a topological class $\alpha$ on $\naka(\vi,\w)= \mu^{-1}_{\vi,\w}(0)^{G_{\vi}-ss}/G_{\vi}$ to a class $\bm\psi(\alpha)$ on the ambient stack $\RG(\vi,\w)$. As for the ordinary stable envelope, the non-abelian extension problem is characterized by a support condition, namely by requiring the class $\bm\psi(\alpha)$ to be supported on $[\mu^{-1}_{\vi,\w}(0)/G_{\vi}]\subset\RG(\vi,\w)$. The map $\bm\psi$ can be built inductively using a stratification of the unstable locus of $\mu^{-1}_{\vi,\w}(0)$ \cite{okounkov2020inductiveII}. This makes the analogy with the abelian stable envelope even more transparent since the strata of the unstable locus can be described in terms of attracting sets for the $G_\vi$-action.

As shown by Aganagic and Okounkov in \cite{aganagic2017quasimap}, non-abelian stable envelopes of Nakajima varieties can be reduced to the abelian case. This approach slightly hides the geometric intuition behind the abovementioned extension problem but allows the known tools developed for ordinary stable envelopes to transfer to this non-abelian setting. For this reason, in this article, we focus on Aganagic-Okounkov's approach\footnote{Although AO worked in K-theory, their construction easily adapts in ordinary cohomology.} to non-abelian stable envelopes.

The key idea behind the construction of $\bm\psi$ is to enlarge the space $\mu^{-1}_{\vi,\w}(0)^{G_{\vi}-ss}$ in such a way that it hosts a copy of $T^*\Rep(\vi,\w)$. To this end, consider the space of representations $T^*\Rep(\vi,\w+\vi)$, which we visualize vertex-wise as 
\begin{equation}
    \label{visualization extra framing}
\begin{tikzcd}
V_i\arrow[d, shift left=.75ex, "i_{i}"]\arrow[drr, shift left=.75ex, "\phi_i"] & & \\
W_i\arrow[u, shift left=.75ex, "j_{i}"] &  & V_i\arrow[ull, shift left=.75ex, "\psi_i"]
\end{tikzcd}
\end{equation}
Notice that the framing extension produces an extra action of a copy of $G_\vi$ on the framing, which we denote by $G_\vi'$.
Define 
\[
\mu^{-1}_{\vi,\w+\vi}(0)^{iso}\subset \mu^{-1}_{\vi,\w+\vi}(0)
\]
as the open subset of $\mu^{-1}_{\vi,\w+\vi}(0)$ determined by the requirement that for every $i\in I $ either $\phi_i$ or $\psi_i$ is an isomorphism, in such a way that we have a factorization 
\[
\mu^{-1}_{\vi,\w+\vi}(0)^{iso}\subset \mu^{-1}_{\vi,\w+\vi}(0)^{G_{\vi}-ss}\subset \mu^{-1}_{\vi,\w+\vi}(0).
\]
For simplicity, we assume from now on that the stability condition is such that $\phi_i$ is an isomorphism for every $i\in I$. It is straightforward how to modify the construction below to deal with the most general case. 

\subsubsection{}
\label{section construction wrong direciton map}

Choose an isomorphism $\phi_i$ for all $i\in I$ and consider the $G_\vi\times T_\w$-equivariant inclusion 
\[
\iota_\phi : T^*\Rep(\vi,\w)\to \mu^{-1}_{\vi,\w+\vi}(0)^{iso} \qquad (x,y,i,j)\to (x,y,i,j,\phi, \chi(x,y,i,j))
\]
where $\chi_i(x,y,i,j)=-\mu_{\vi,\w}(x,y,i,j)\circ \phi_i^{-1}$. Notice that this formula for $\chi$ is forced by the moment map condition on the codomain.
To get rid of the dependence on the isomorphisms $\phi_i$, we compose $\iota_\phi$ with the quotient by $G_{\vi}'$:
\[
T^*\Rep(\vi,\w)\to \mu^{-1}_{\vi,\w+\vi}(0)^{iso}/G_{\vi}'.
\]
Quotienting further by $G_{\vi}$, we get a $T_\w$-equivariant map
\[
\bm\iota: \RG(\vi,\w)\to [\naka(\vi,\w+\vi)/G_{\vi}'].
\]

\subsubsection{}

Passing to cohomology, we get a map 
\[
\bm\iota^*: H_{T_\w}([\naka(\vi,\w+\vi)/G_{\vi}'])\to H_{T_\w}(\RG(\vi,\w)).
\]
Notice that this map factors as 
\begin{equation*}
    \begin{tikzcd}
    H_{T_\w}([\naka(\vi,\w+\vi)/G_{\vi}'])\arrow[d, equal]\arrow[r, "\bm\iota^*"] &H_{T_\w}(\RG(\vi,\w))\arrow[ddd, equal]\\
    H_{T_\w\times G_{\vi}'}(\naka(\vi,\w+\vi))\arrow[d, equal] & \\
    H_{T_\w\times G_{\vi}\times G_{\vi}'}(\mu^{-1}_{\vi,\w+\vi}(0)^{G_{\vi}-ss})\arrow[d, "\text{res}"] & \\
    H_{T_\w\times G_{\vi}\times G_{\vi}'}(\mu^{-1}_{\vi,\w+\vi}(0)^{iso})]\arrow[r, "\iota_\phi^*"] &  H_{T_\w\times G_{\vi}}(T^*\Rep(\vi,\w))
    \end{tikzcd}
\end{equation*}
Because of the $G_\vi'$ action, the dependence of the bottom horizontal map on $\phi$ is only apparent. Moreover, since by definition of $\mu^{-1}_{\vi,\w+\vi}(0)^{iso}$ the maps $\phi_i$ as in \eqref{visualization extra framing} are isomorphisms, the Chern roots of $G_\vi$ are identified with those of $G_{\vi}'$ in 
\[
H_{T_\w\times G_{\vi}\times G_{\vi}'}(\mu^{-1}_{\vi,\w+\vi}(0)^{iso}).
\]

\subsubsection{}

To complete the construction of $\bm\psi$ we need a map 
\begin{equation}
\label{desired map}
    H_{T_\w}(\naka(\vi,\w))\to H_{T_\w\times G_\vi'}(\naka(\vi,\w+\vi))=H_{T_\w}([\naka(\vi,\w+\vi)/G_{\vi}']).
\end{equation}
Once again this map is provided by stable envelopes. First of all, notice that the trivial identity
\[
\naka(\vi,\w)=\naka(\vi,\w)\times \naka(0,\vi)
\]
allows us to think of $\naka(\vi,\w)$ as a fixed component of $\naka(\vi,\w+\vi)$ by the action of the one dimensional subtorus $A\subset A_{\w+\vi}$ acting on $W_i$ with weight one and trivially elsewhere. 
Now consider the stable envelope map
\begin{equation}
    \label{stable envelope for weight functions}
    \text{Stab}_{\lbrace a< 0\rbrace }: H_{T_\w\times G_\vi'}(\naka(\vi,\w)\times \naka(0,\vi))\to H_{T_\w\times G_\vi'}(\naka(\vi,\w+\vi)).
\end{equation}
for the distinguished component $\naka(\vi,\w)\times \naka(0,\vi)\subset \naka(\vi,\w+\vi)^{\Ci^\times}$. Notice that although we have defined stable envelopes for the action of some torus $T$ on $\naka(\vi,\w)$, the correspondence $\text{Stab}_{\lbrace a< 0\rbrace}$ actually gives a map even in $T_\w\times G_\vi'$ equivariant cohomology. This follows from the explicit construction of $\text{Stab}_{\lbrace a< 0\rbrace}$ as a Lagrangian correspondence, and the fact that the action of $A$ commutes with the one of $T_\w\times G_\vi'$.

Finally, noticing that $\naka(0,\vi)$ is a singleton and composing \eqref{stable envelope for weight functions} with 
\[
 H_{T_\w}(\naka(\vi,\w))\hookrightarrow  H_{T_\w\times G_\vi'}(\naka(\vi,\w))=H_{T_\w\times G_\vi'}(\naka(\vi,\w)\times \naka(0,\vi)),
\]
we get the desired map \eqref{desired map}.

\subsubsection{}

Putting it all together, we get the sought-after non-abelian stable envelope map
\[
\bm\psi: H_{T_\w}(\naka(\vi,\w))\to H_{T_\w\times G_{\vi}'}(\naka(\vi,\w+\vi))= H_{T_\w}([\naka(\vi,\w+\vi)/G_{\vi}'])\xrightarrow{\bm\iota^*} H_{T_\w}(\RG(\vi,\w)).
\]
 The support condition of \eqref{stable envelope for weight functions} implies the following:
\begin{lma}[{\cite[Proposition 1,3]{aganagic2017quasimap}}]
\label{Lemma section up to a factor}
The map $\bm\psi$ is supported on $[\mu_{\vi,\w}^{-1}(0)/G_{\vi}]\subset [T^*\Rep(\vi,\w)/G_{\vi}]$. Moreover, the restriction $\bm{j}^*(\bm\psi(\alpha))$ is divisible by $e({\hbar}\g_{\vi})$
and 
\[
\alpha= \bm{j}^*(\bm\psi(\alpha))/e({\hbar}\g_{\vi}).
\]
\end{lma}

\subsubsection{}

We can now state the main theorem of this article.

\begin{thm}
\label{main theorem}
The map $\bm\psi: \HhQ\to \CoHAQ_{\tau}$ is an injective morphism of $\N^I\times \N^I$-graded $\Bbbk$-algebras.
\end{thm}

In words, the theorem above says that the map $\bm\psi$ realizes the equivariant cohomology of Nakajima varieties as a subalgebra of the framed CoHA, and on this subalgebra CoHA's multiplication is identified with taking stable envelopes. 

\begin{remq}
The map $\bm\psi$ is not compatible with the cup products $(\alpha,\beta)\mapsto \alpha\beta$ on its domain and codomain. This can be easily seen as follows. By contradiction, fix a non-trivial graded component $\HhQ(\vi,\w)$ with $\vi\neq 0$ and assume that $\bm\psi(\alpha\beta)=\bm\psi(\alpha)\bm\psi(\beta)$. Composing with $\bm j^*$, which is clearly compatible with the cup product, and applying Lemma \ref{Lemma section up to a factor}, we get 
\[
e({\hbar}\g_{\vi}) \alpha\beta= \bm j^*\bm\psi(\alpha\beta)= \bm j^* (\bm\psi(\alpha)\bm\psi(\beta))=\bm j^*\bm\psi(\alpha)\bm j^*\bm\psi(\beta)=e({\hbar}\g_{\vi})^2 \alpha\beta
\]
for all $\alpha$ and $\beta$. This equation would imply $e({\hbar}\g_{\vi})^2=e({\hbar}\g_{\vi})$, which is a contradiction since $e({\hbar}\g_{\vi})\neq 0$.
\end{remq}


\subsection{Proof of Theorem \ref{main theorem}: Step one}

First of all, notice that injectivity follows at once from Lemma \ref{Lemma section up to a factor}, so it remains to prove that $\bm\psi$ is a morphism of graded algebras. The proof can be conceptually divided into two parts. In the first part, we prove an ``abelianized'' version of the statement and in the second we apply Theorem \ref{theorem abelianization CoHA} and Theorem \ref{theorem abelianization stable envelopes} to reduce the claim to its abelianized version.

\subsubsection{}

Define
\[
\Hh_{\text{ab}}:=\bigoplus_{\vi,\w\in \N^I} \Hh_{\text{ab}}(\vi,\w)
\]
where 
\[
\Hh_{\text{ab}}(\vi,\w):=H_{T_\w}(\naka_S(\vi,\w)).
\]
With the same arguments of Section \ref{Section framed semistable CoHA}, one proves that the stable envelope maps
\[
\text{Stab}_{\lbrace a_1<a_2\rbrace}: \Hh_{\text{ab}}(\vi_1,\w_1)\otimes_\Bbbk \Hh_{\text{ab}}(\vi_2,\w_2)\to \Hh_{\text{ab}}(\vi_1+\vi_2,\w_1+\w_2)
\]
give $\Hh_{\text{ab}}$ the structure of a $\N^I\times \N^I$-graded associative $\Bbbk$-algebra.

\subsubsection{}
\label{construction of abelianized psi}
We now mimic the construction of Section \ref{section construction wrong direciton map} to define an abelianized version of $\bm\psi$, i.e. a map
\[
\bm{\psi}_{\text{ab}}: \Hh_{\text{ab}}\to \CoHA_{\text{ab},\tau}.
\]
Let $S\subset G_{\vi}$ be a maximal torus. As in section \ref{section construction wrong direciton map}, choose an isomorphism $\phi_i$ for all $i\in I$ and consider $S\times T_\w$-equivariant inclusion 
\[
\iota_{\text{ab},\phi} : T^*\Rep(\vi,\w)\oplus {\hbar}\s^\perp\to \mu^{-1}_{\vi,\w+\vi}(\s^\perp)^{iso}
\]
by 
\[
(x,y,i,j,v)\mapsto (x,y,i,j,\phi, \chi(x,y,i,j)+v\circ \phi^{-1})
\]
where $\chi_i(x,y,i,j)=-\mu_{\vi,\w}(x,y,i,j)\circ \phi_i^{-1}$.

To get rid of the dependence by the isomorphisms $\phi_i$, we compose $\iota_{\text{ab},\phi}$ with the quotient by $G_{\vi}'$ to get a map
\[
\bm\iota_{\text{ab}} : T^*\Rep(\vi,\w)\oplus {\hbar}\s^\perp\to \mu^{-1}_{\vi,\w+\vi}(\s^\perp)^{iso}/G_{\vi}'.
\]
Quotienting further by $S$, we get a $T_\w$-equivariant map
\[
\bm\iota_{\text{ab}}: [T^*\Rep(\vi,\w)\oplus  {\hbar}\s^\perp/S]\to [\naka_S(\vi,\w+\vi)/G_{\vi}'],
\]
which we denote in the same way.


\subsubsection{}

Passing to cohomology, we get a map 
\[
\bm{\iota}_{\text{ab}}^*: H_{T_\w}([\naka_S(\vi,\w+\vi)/G_{\vi}'])\to H_{T_\w}( [T^*\Rep(\vi,\w)\oplus  {\hbar}\s^\perp/S])
\]
that, composed with the top horizontal map of the diagram
\[
\begin{tikzcd}
 H_{T_\w}(\naka_S(\vi,\w))\arrow[d, hook]\arrow[r] & H_{T_\w}(\left[\naka_S(\vi,\w+\vi)/G_{\vi'}\right])\arrow[d, equal]\\
 H_{T_\w\times G_\vi'}(\naka_S(\vi,\w)\times X_S(0,\vi))\arrow[r, "\text{Stab}_{\lbrace a< 0\rbrace }"]& H_{T_\w\times G_\vi'}(\naka_S(\vi,\w+\vi))
\end{tikzcd}
\]
gives the sought-after map 
\[
\bm{\psi}_{\text{ab}}: \Hh_{\text{ab}}(\vi,\w)= H_{T_\w}(\naka_S(\vi,\w))\to H_{T_\w}( [T^*\Rep(\vi,\w)\oplus  {\hbar}\s^\perp/S])=\CoHA_{\text{ab},\tau}(\vi,\w)
\]
\begin{lma}
The map $\bm{\psi}_{\text{ab}}$ is an algebra morphism.
\end{lma}
\begin{proof}
Let $\vi_1+\vi_2=\w$ and $\w_1+\w_2=\w$. We have to show that the diagram
\begin{equation*}
    \begin{tikzcd}
      \Hh_{\text{ab}}(\vi_1,\w_1)\otimes_\Bbbk \Hh_{\text{ab}}(\vi_2,\w_2)\arrow[d, "\text{Stab}_{\lbrace a_1<a_2\rbrace}"]\arrow[rr, "\bm{\psi}_{\text{ab}}\otimes \bm{\psi}_{\text{ab}}"]& & \CoHA_{\text{ab},\tau}(\vi_1,\w_1)\otimes_\Bbbk \CoHA_{\text{ab},\tau}(\vi_2,\w_2)\arrow[d, "\mathsf{m}_{\text{ab},\tau}"]\\
      \Hh_{\text{ab}}(\vi,\w) \arrow[rr, "\bm{\psi}_{\text{ab}}"] & &  \CoHA_{\text{ab},\tau}(\vi,\w)
    \end{tikzcd}
\end{equation*}
commutes.
Let $\gamma=\gamma_1\otimes \gamma_2\in \Hh_{\text{ab}}(\vi_1,\w_2)\otimes_\Bbbk \Hh_{\text{ab}}(\vi_2,\w_2)$ and let $p_\gamma=p_{\gamma_1}\otimes p_{\gamma_2}$ be a tautological presentation. Applying Proposition \ref{propn explicit formula stab hypertoric} and the definition of $\bm{\psi}_{\text{ab}}$ we get
\begin{align*}
    \bm{\psi}_{\text{ab}}\circ \text{Stab}_{\lbrace a_1<a_2\rbrace}(\gamma)
    &=\bm{\psi}_{\text{ab}}(e(T^*\Rep(\vi,\w)[-1])p_\gamma)\\
    &=\bm{\iota}_{\text{ab}}^*\circ \text{Stab}_{\lbrace a< 0\rbrace }(e(T^*\Rep(\vi,\w)[-1])p_\gamma)\\
    &= \bm{\iota}_{\text{ab}}^*(e({\hbar}\Hom(V',V))e(T^*\Rep(\vi,\w)[-1])p_\gamma)\\
    &=e({\hbar}\Hom(V,V))e(T^*\Rep(\vi,\w)[-1])p_\gamma\\
    &= e({\hbar}\g_\vi)e(T^*\Rep(\vi,\w)[-1])p_\gamma
\end{align*}
Similarly, applying Proposition \ref{propn explicit formula stab hypertoric} and equation \eqref{explicit formula multiplication abelianized CoHA} we get
\begin{align*}
    \mathsf{m}_{\text{ab},\tau}\circ (\bm{\psi}_{\text{ab}}\otimes \bm{\psi}_{\text{ab}})(\gamma_1\otimes \gamma_2)
    &= \mathsf{m}_{\text{ab},\tau}\circ ((\bm{\iota}_{\text{ab}}^*\circ \text{Stab}_{\lbrace a_1< 0\rbrace }(p_{\gamma_1}))\otimes( \bm{\iota}_{\text{ab}}^*\circ \text{Stab}_{\lbrace a_2< 0\rbrace }(p_{\gamma_2})))\\
    &=\mathsf{m}_{\text{ab},\tau} (e({\hbar}\Hom(V_1,V_1))p_{\gamma_1}\otimes e({\hbar}\Hom(V_2,V_2))p_{\gamma_2})\\
    &=e(T^*\Rep(\vi,\w)[-1])e({\hbar}\g_{\vi}[-1])e({\hbar}\g_{\vi}[1])e({\hbar}\g_{\vi_1})e({\hbar}\g_{\vi_2})p_{\gamma_1}p_{\gamma_2}
\end{align*}
Since $\g_\vi=\g_{\vi_1}\oplus \g_{\vi_2}\oplus\g_{\vi}[-1]\oplus \g_{\vi}[1]$ and $p_\gamma=p_{\gamma_1}p_{\gamma_2}$, the result follows. 
\end{proof}


\subsection{Proof of Theorem \ref{main theorem}: Step two}

\subsubsection{}

We can now complete the proof by combining abelianization of stable envelopes and abelianization of CoHA. Applying the first part of Theorem \ref{theorem abelianization CoHA} and Theorem \ref{theorem abelianization stable envelopes}, one gets the commutative diagram
\[
\begin{tikzcd}
      & \Hh_{\text{ab}}(\vi_1,\w_2)\otimes_\Bbbk \Hh_{\text{ab}}(\vi_2,\w_2)\arrow[dd, shift left=1.5ex, "\text{Stab}_{\lbrace a_1<a_2\rbrace}" {yshift=-20pt}]\arrow[r, "\bm{\psi}_{\text{ab}}\otimes \bm{\psi}_{\text{ab}}"] & \CoHA_{\text{ab},\tau}(\vi_1,\w_1)\otimes_\Bbbk \CoHA_{\text{ab},\tau}(\vi_2,\w_2)\arrow[dd, "\mathsf{m}_{\text{ab},\tau}"]\\
      H_{T_{\w_1}}(\naka_{\bi_1})\otimes_\Bbbk H_{T_{\w_2}}(\naka_{\bi_2})\arrow[ur]\arrow[dd, "\text{Stab}_{\lbrace a_1<a_2\rbrace}^{\bi}(\naka_{\bi_1}\times \naka_{\bi_2})"] & H_{T_{\w_1}}^*(\RG_{\bi_1})\otimes_\Bbbk H_{T_{\w_2}}^*(\RG_{\bi_2})\arrow[ur]\arrow[dd, swap, shift right=1ex, "\mathsf{m}_{\text{ab},\tau}^{\bi}" {yshift=-15pt}]& \\
      & \Hh_{\text{ab}}(\vi,\w) \arrow[r, "\bm{\psi}_{\text{ab}}"] &  \CoHA_{\text{ab},\tau}(\vi,\w)\\\
       H_{T_{\w}}(\naka_{\bi})\arrow[ur]  & H_{T_\w}(\RG_{\bi})\arrow[ur] &
    \end{tikzcd}
\]
where $\naka_{\bi_i}:=\mu_{\vi,\w}^{-1}(\bi_i^\perp)^{S-ss}/S_i$. We now construct horizontal maps that fill the two missing edges of the cube above. 
\begin{lma}
There exist dashed maps
\[
\begin{tikzcd}
      & \Hh_{\text{ab}}(\vi_1,\w_2)\otimes_\Bbbk \Hh_{\text{ab}}(\vi_2,\w_2)\arrow[dd, shift left=1.5ex, "\text{Stab}_{\lbrace a_1<a_2\rbrace}" {yshift=-20pt}]\arrow[r, "\bm{\psi}_{\text{ab}}\otimes \bm{\psi}_{\text{ab}}"] & \CoHA_{\text{ab},\tau}(\vi_1,\w_1)\otimes_\Bbbk \CoHA_{\text{ab},\tau}(\vi_2,\w_2)\arrow[dd, "\mathsf{m}_{\text{ab},\tau}"]\\
      H_{T_{\w_1}}(\naka_{\bi_1})\otimes_\Bbbk H_{T_{\w_2}}(\naka_{\bi_2})\arrow[ur]\arrow[r, dashrightarrow]\arrow[dd, "\text{Stab}_{\lbrace a_1<a_2\rbrace}^{\bi}(\naka_{\bi_1}\times \naka_{\bi_2})"] & H_{T_{\w_1}}^*(\RG_{\bi_1})\otimes_\Bbbk H_{T_{\w_2}}^*(\RG_{\bi_2})\arrow[ur]\arrow[dd, swap, shift right=1ex, "\mathsf{m}_{\text{ab},\tau}^{\bi}" {yshift=-15pt}]& \\
      & \Hh_{\text{ab}}(\vi,\w) \arrow[r, "\bm{\psi}_{\text{ab}}"] &  \CoHA_{\text{ab},\tau}(\vi,\w)\\
       H_{T_{\w}}(\naka_{\bi})\arrow[ur]\arrow[r, dashrightarrow]  & H_{T_\w}(\RG_{\bi})\arrow[ur] &
    \end{tikzcd}
\]
making the diagram commute. 
\end{lma}
\begin{proof}
We begin with the lower edge. Unraveling the definition of $\bm{\psi}_{\text{ab}}$, one sees that we need to find a dashed map that makes the following diagram commute 
\begin{equation}
\label{sought after dashed arrow}
\begin{tikzcd}
  H_{T_{\w}}(\naka_S(\vi,\w))\arrow[r, "\text{Stab}_{\lbrace a<0\rbrace}"] & H_{T_\w\times G_{\vi'}}(\naka_S(\vi,\w+\vi))\arrow[r, equal] &H_{T_\w}(\left[\naka_S(\vi,\w+\vi)/G_{\vi}'\right]) \arrow[r, "\bm{\iota}_{\text{ab}}^*"]&  \CoHA_{\text{ab},\tau}(\vi,\w)\\
  H_{T_{\w}}(\naka_{\bi})\arrow[rrr, dashrightarrow]\arrow[u, "(j_-)_*"] & & & H_{T_\w}(\RG_{\bi})\arrow[u, "(i_-)_*"]
\end{tikzcd}
\end{equation}
The fiber squares\footnote{Notice that the central square is Cartesian because the action of $G_{\vi}'$ is free on the $iso$-locus.} 
\begin{equation}
\label{fiber squares step two proof}
    \begin{tikzcd}T^*\Rep(\vi,\w)\oplus {\hbar}\n \arrow[d, hookrightarrow, "i_-"]\arrow[r, hookrightarrow,"\iota_{\bi,\phi}"] &\mu^{-1}_{\vi,\w+\vi}(\bi^\perp)^{iso}\arrow[d, hookrightarrow]\arrow[r] & \mu^{-1}_{\vi,\w+\vi}(\bi^\perp)^{iso}/G_{\vi}'\arrow[d, hookrightarrow]\arrow[r, hookrightarrow] & \left[\mu^{-1}_{\vi,\w+\vi}(\bi^\perp)^{S-ss}/G_{\vi}'\right]\arrow[d, hookrightarrow, "\tilde j_-"]\\
   T^*\Rep(\vi,\w)\oplus {\hbar}\s^\perp \arrow[r, hookrightarrow, "\iota_{\text{ab},\phi}"] &\mu^{-1}_{\vi,\w+\vi}(\s^\perp)^{iso} \arrow[r]& \mu^{-1}_{\vi,\w+\vi}(\s^\perp)^{iso}/G_{\vi}'\arrow[r, hookrightarrow] & \left[\mu^{-1}_{\vi,\w+\vi}(\s^\perp)^{S-ss}/G_{\vi}'\right]
   \end{tikzcd}
\end{equation}
gives the following commutative diagram when passing to $S\times T_\w$-equivariant cohomology
\[
\begin{tikzcd}
   H_{T_\w\times G_{\vi'}}(\naka_S(\vi,\w+\vi))\arrow[r, equal] &H_{T_\w}(\left[\naka_S(\vi,\w+\vi)/G_{\vi}'\right]) \arrow[r, "\bm{\iota}_{\text{ab}}^*"]&  \CoHA_{\text{ab},\tau}(\vi,\w)\\
   H_{T_\w\times G_{\vi}'}(\left[\mu^{-1}_{\vi,\w+\vi}(\bi^\perp)^{S-ss}/S\right])\arrow[r, equal]\arrow[u, "(\tilde{j}_-)_*"]& H_{T_\w}(\left[\mu^{-1}_{\vi,\w+\vi}(\bi^\perp)^{S-ss}/S\times G_{\vi}'\right])\arrow[r, "\bm{\iota}_\bi^*"]\arrow[u, "(\tilde{j}_-)_*"]& H_{T_\w}(\RG_{\bi})\arrow[u, "(i_-)_*"]
\end{tikzcd}
\]
The tilde in the map 
\[\tilde j_-: \mu^{-1}_{\vi,\w+\vi}(\bi^\perp)^{S-ss}\to \mu^{-1}_{\vi,\w+\vi}(\s^\perp)^{S-ss}
\] is chosen to distinguish it from 
\[j_-: \mu^{-1}_{\vi,\w}(\bi^\perp)^{S-ss}\to \mu^{-1}_{\vi,\w}(\s^\perp)^{S-ss}.
\]
Applying again the first part of Theorem \ref{theorem abelianization stable envelopes}, now with $F_S=\naka_S(\vi,\w)=\naka_S(\vi,\w)\times \naka_S(0,\w)$, we get the commutative square 
\begin{equation*}
    \begin{tikzcd}
      H_{T_{\w}}(\naka_S(\vi,\w))\arrow[r, "\text{Stab}_{\lbrace a<0\rbrace}"] & H_{T_\w\times G_{\vi'}}(\naka_S(\vi,\w+\vi)\\
      H_{T_{\w}}(\naka_{\bi})\arrow[r]\arrow[u, "(j_-)_*"] & H_{T_\w\times G_{\vi}'}(\left[\mu^{-1}_{\vi,\w+\vi}(\bi^\perp)^{S-ss}/S\right])\arrow[u, "(\tilde j_-)_*"]
    \end{tikzcd}
\end{equation*}
that, glued with the previous diagram along $(\tilde{j}_-)_*$, completes the construction of the dashed arrow in diagram \eqref{sought after dashed arrow}. 
The construction of the other edge of the cube is completely analogous and consists of performing the construction above twice, once for each side of the tensor products. Overall, we have shown that all the sides of the cube commute, except from the outermost face (i.e. the one where both the dashed arrows lie). However, commutativity of this last face follows from the following lemma together with the observation that the map $(i_-)_*$ is injective because it is given by multiplication by the Euler class of the normal bundle to the inclusion $i_-$ and its source is an integral domain.
\end{proof}
\begin{lma}
Consider the following diagram of modules:
\[
\begin{tikzcd}
  & \widehat{A}\arrow[dd, "\hat f" {yshift=-10pt}] \arrow[rr, "\hat g" {xshift=-10pt}] & & \widehat{B} \arrow [dd, "\hat h"]\\
  A\arrow[dd, "f"]\arrow[ur, "\alpha"]\arrow[rr, "g" {xshift=-10pt}] & & B\arrow[dd, "h" {yshift=10pt}]\arrow[ur, "\beta"] &\\
  & \widehat{C}\arrow[rr, "\hat k" {xshift=-10pt}] & & \widehat{D} \\
  C \arrow[rr, "k" {xshift=-10pt}]\arrow[ur, "\gamma"] & & D\arrow[ur, "\delta"]
\end{tikzcd}
\]
Assume that all the faces commute except the outermost and that the map $\delta$ is injective. Then also the outermost face commute.
\begin{proof}
By injectivity of $\delta$, it suffices to check that $\delta \circ k\circ f =\delta \circ h\circ g$. Indeed, we have
\[
\delta \circ h\circ g= \hat h \circ \beta \circ g=\hat h\circ \hat g\circ \alpha=\hat k\circ \hat f\circ \alpha =\hat k \circ \gamma \circ f= \delta \circ k\circ f.
\]
\end{proof}
\end{lma}

\subsubsection{}
Applying now the second parts of Theorem \ref{theorem abelianization CoHA} and Theorem \ref{theorem abelianization stable envelopes}, we get the solid diagram
\begin{equation}
\label{last big cube}
\begin{tikzcd}
& H_{T_{\w_1}}(\naka_{\bi_1})\otimes_\Bbbk H_{T_{\w_2}}(\naka_{\bi_2})\arrow[dl]\arrow[dd, shift left=1.5ex, "\text{Stab}_{\lbrace a_1<a_2\rbrace}^{\bi}(\naka_{\bi_1}\times \naka_{\bi_2})" {yshift=-20pt}]\arrow[r, dashrightarrow] & H_{T_{\w_1}}^*(\RG_{\bi_1})\otimes_\Bbbk H_{T_{\w_2}}^*(\RG_{\bi_2})\arrow[dd, "\mathsf{m}_{\text{ab},\tau}^{\bi}"]\arrow[dl] &\\
\Hh(\vi_1,\w_1)\otimes_\Bbbk \Hh(\vi_2,\w_2)\arrow[r, "\bm{\psi}\otimes \bm{\psi}"]\arrow[dd, swap,  "\text{Stab}_{\lbrace a_1<a_2\rbrace}"]& \CoHA(\vi_1,\w_1)\otimes_\Bbbk \CoHA(\vi_2,\w_2)\arrow[dd, swap, shift right=0.7ex, "\mathsf{m}_{\tau}" {yshift=-15pt}] &\\
& H_{T_{\w}}(\naka_{\bi})\arrow[dl, swap, "\pi_*\circ (\Pi^*)^{-1}\circ (j_+)^* "]\arrow[r, dashrightarrow] & H_{T_\w}(\RG_{\bi})\arrow[dl, "q_*\circ (Q^*)^{-1}\circ (i_+)^* "]&\\
\Hh(\vi,\w)\arrow[r, "\bm{\psi}"]& \CoHA(\vi,\w) & 
\end{tikzcd}
\end{equation}
while the dashed arrows are those constructed in the previous section. Commutativity of the outermost face is our final goal. The top left oblique map is $(\pi_{1}\times\pi_{2})_*\circ ((\Pi_1\times \Pi_2)^*)^{-1}\circ (j_{1,+}\times  j_{2,+})^*$ and the top right oblique map is $(q_{1}\times q_{2})_*\circ ((Q_1\times Q_2)^*)^{-1}\circ (i_{1,+}\times  i_{2,+})^*$. By construction, the two lateral sides and the back of the diagram commute. Moreover, the oblique arrows are surjective, so to complete the proof it suffices to show that the two remaining faces, the top and the bottom of the cube, are commutative. Therefore, the following lemma concludes the proof.
\begin{lma}
The top and bottom faces of the cubic diagram above commute. 
\end{lma}
\begin{proof}
As for the previous lemma, the top face is just two copies of the bottom face tensored together, so it suffices to prove the commutativity of the bottom face. 
Recall that the dashed arrow in the bottom face is built from diagram \eqref{fiber squares step two proof}, on which we now build up the following commutative diagram
\[
\begin{tikzcd}
T^*\Rep(\vi,\w)\oplus {\hbar}\n \arrow[r, "\iota_{\bi,\phi}"] &\mu^{-1}_{\vi,\w+\vi}(\bi^\perp)^{iso}\arrow[r] & \mu^{-1}_{\vi,\w+\vi}(\bi^\perp)^{iso}/G_{\vi}' \arrow[r] & \left[\mu^{-1}_{\vi,\w+\vi}(\bi^\perp)^{S-ss}/ G_{\vi}'\right]\\
    T^*\Rep(\vi,\w)\arrow[r, "\iota_{\phi}"]\arrow[u, "i_+"] &\mu^{-1}_{\vi,\w+\vi}(0)^{iso}\arrow[r]\arrow[u]& \mu^{-1}_{\vi,\w+\vi}(0)^{iso}/ G_{\vi}'\arrow[r]\arrow[u]& \left[\mu^{-1}_{\vi,\w+\vi}(0)^{G_{\vi}-ss}/ G_{\vi}'\right] \arrow[u, "\tilde j_+"]
\end{tikzcd}
\]
Passing to $S\times T_\w$-equivariant cohomology, we get the commutative diagram
\[
\begin{tikzcd}
H_{T_\w\times G_{\vi}'}(\mu^{-1}_{\vi,\w+\vi}(\bi^\perp)^{S-ss}/S)\arrow[r, equal]\arrow[d, "\tilde j_+^*"]&H_{T_\w}(\left[\mu^{-1}_{\vi,\w+\vi}(\bi^\perp)^{S-ss}/ S\times G_{\vi}'\right])\arrow[r, "\iota_\bi^*"] &H_{T_\w}(\left[T^*\Rep(\vi,\w)\oplus {\hbar}\n/ S\right])\arrow[d, "i_+^*"]  \\
H_{T_\w\times G_{\vi}'}(\mu^{-1}_{\vi,\w+\vi}(0)^{G_{\vi}-ss}/S)\arrow[r, equal]& H_{T_\w}(\left[\mu^{-1}_{\vi,\w+\vi}(0)^{G_{\vi}-ss}/ S\times G_{\vi}'\right])\arrow[r] & H_{T_\w}(\left[T^*\Rep(\vi,\w)/ S\right])
\end{tikzcd}
\]
Taking quotients by $B$ and $G_\vi$ we can enlarge the previous commutative diagram as follows
\[
\begin{tikzcd}
H_{T_\w\times G_{\vi}'}(\mu^{-1}_{\vi,\w+\vi}(\bi^\perp)^{S-ss}/S)\arrow[r, equal]\arrow[d, "\tilde j_+^*"]&H_{T_\w}(\left[\mu^{-1}_{\vi,\w+\vi}(\bi^\perp)^{S-ss}/ S\times G_{\vi}'\right])\arrow[r, "\iota_\bi^*"] &H_{T_\w}(\left[T^*\Rep(\vi,\w)\oplus h\bi^\perp/ S\right])\arrow[d, "i_+^*"]  \\
H_{T_\w\times G_{\vi}'}(\mu^{-1}_{\vi,\w+\vi}(0)^{G_{\vi}-ss}/S)\arrow[r, equal]& H_{T_\w}(\left[\mu^{-1}_{\vi,\w+\vi}(0)^{G_{\vi}-ss}/ S\times G_{\vi}'\right])\arrow[r] & H_{T_\w}(\left[T^*\Rep(\vi,\w)/ S\right])\\
H_{T_\w\times G_{\vi}'}(\mu^{-1}_{\vi,\w+\vi}(0)^{G_{\vi}-ss}/B)\arrow[d, "\tilde\pi_*"]\arrow[u, swap, "\tilde \Pi^*"]\arrow[r, equal]&H_{T_\w}(\left[\mu^{-1}_{\vi,\w+\vi}(0)^{G_{\vi}-ss}/ B\times G_{\vi}'\right])\arrow[r] &H_{T_\w}(\left[T^*\Rep(\vi,\w)/ B\right])\arrow[u, swap, "Q^*"]\arrow[d, "q_*"]\\
H_{T_\w\times G_{\vi}'}(\mu^{-1}_{\vi,\w+\vi}(0)^{G_{\vi}-ss}/G_{\vi})\arrow[r, equal]&H_{T_\w}(\left[\mu^{-1}_{\vi,\w+\vi}(0)^{G_{\vi}-ss}/ G_{\vi}\times G_{\vi}'\right])\arrow[r, "\iota^*"] &H_{T_\w}(\left[T^*\Rep(\vi,\w)/ G_{\vi}\right])
\end{tikzcd}
\]
Notice that the right side of the square consists of the maps appearing in the lower right side of diagram \eqref{last big cube}. Similarly, the maps on the left side of the square are the ones appearing in the diagram 
\[
\begin{tikzcd}
H_{T_\w\times G_\vi'}(\naka(\vi,\w))\arrow[d, "\text{Stab}_{\lbrace a<0\rbrace}"]  & & & H_{T_\w\times G_\vi'}(\naka_\bi)\arrow[d, "\text{Stab}_{\lbrace a<0\rbrace}^\bi"]\arrow[lll, swap,  "\pi_*\circ (\Pi^*)^{-1}\circ j_+^* "] \\
H_{T_\w\times G_\vi'}(\naka(\vi,\w+\vi)) & & &  H_{T_\w\times G_{\vi}'}(\mu_{\vi,\w+\vi}^{-1}(\bi^\perp)^{S-ss}/S) \arrow[lll, swap,  "\tilde\pi_*\circ (\tilde\Pi^*)^{-1}\circ \tilde j_+^*"]
\end{tikzcd}
\]
As before, this diagram is the second abelianization diagram for the fixed component $\naka(\vi,\w)=\naka(\vi,\w)\times \naka(0,\vi)\subset \naka(\vi,\w+\vi)$, see Theorem \ref{theorem abelianization stable envelopes}. 
Combining the last diagram with the outer frame of the penultimate one, we get the following commutative diagram
\[
\begin{tikzcd}
H_{T_\w}(\naka_\bi)\arrow[d, , swap, "\pi_*\circ (\Pi^*)^{-1}\circ j_+^* "]\arrow[rrr, bend left, dashrightarrow] \arrow[r]& H_{T_\w\times G_{\vi}'}(\naka_\bi)\arrow[d, swap, "\pi_*\circ (\Pi^*)^{-1}\circ j_+^* "]\arrow[r, "\text{Stab}_{\lbrace a<0\rbrace}^\bi"] &  H_{T_\w\times G_{\vi}'}(\mu^{-1}(\bi^\perp)^{S-ss}/S)\arrow[r, "\iota_\bi^*"]\arrow[d, swap, "\tilde\pi_*\circ (\tilde\Pi^*)^{-1}\circ \tilde j_+^*"] & H_{T_\w}(\RG_\bi)\arrow[d, "q_*\circ (Q^*)^{-1}\circ (i_+)^* "]\\
\Hh(\vi,\w)\arrow[rrr, bend right, "\bm\psi"] \arrow[r] & H_{T_\w\times G_{\vi}'}(\naka(\vi,\w))\arrow[r, swap, "\text{Stab}_{\lbrace a<0 \rbrace}"] &H_{T_\w\times G_\vi'}(\naka(\vi,\w+\vi))\arrow[r, "\iota^*"] & \CoHA(\vi,\w)
\end{tikzcd}
\]
The two unlabelled horizontal maps are simply the canonical change of group maps. The commuting outer frame of this diagram was exactly what we had to prove that commutes. Therefore we are done.
\end{proof}

\subsection{Application: Shuffle product of stable envelopes}

\subsubsection{}
One of the main reasons for Aganagic and Okounkov to define the map
\[
\bm\psi: H_{T_\w}(\naka(\vi,\w))\to H_{T_\w}(\RG(\vi,\w))=\Bbbk[\liea_\w\times \s]^{W_{G_{\vi}}}
\]
was to give a systematic way to produce tautological presentations of stable envelopes
\[
\text{Stab}_{\mathfrak{C}}: H_{T_\w}(\naka(\vi,\w)^A)\to H_{T_\w}(\naka(\vi,\w)).
\]
Indeed, by Lemma \ref{Lemma section up to a factor}, the assignment 
\[
\text{Stab}^{\bm\psi}_{\mathfrak{C}}:=\frac{\bm{\psi}\circ \text{Stab}_{\mathfrak{C}}}{e({\hbar}\g_\vi)}
\]
defines a function 
\[
 \text{Stab}_{\mathfrak{C}}^{\bm\psi}: H_{T_\w}(\naka(\vi,\w)^A)\to H_{T_\w}(\RG(\vi,\w))_{loc}=\Bbbk[\liea_\w\times \s]^{W_{G_{\vi}}}_{loc}
\]
that, restricted to $\naka(\vi,\w)\subset \RG(\vi,\w)$, is an integral class, i.e. defined in the non-localized ring, and it coincides with $\text{Stab}_{\mathfrak{C}}$. 

\subsubsection{}

With this preliminary observation, Theorem \ref{main theorem} can be applied to obtain explicit inductive formulas for the stable envelopes of a Nakajima variety $\naka(\vi,\w)$. 

As in Section \ref{subsection Stable envelopes of Nakajima varieties: combinatorics}, consider a fixed component $F$ for the action of a subtorus
\[
A=\lbrace (a_1,a_2,\dots ,a_k) \mid a_i\in \Ci^\times  \rbrace\hookrightarrow A_\w
\]
associated to a decomposition $\w=\w_1+\w_2\dots+\w_k$ and a chamber 
\[
\mathfrak{C}_{\sigma}=\lbrace a_{\sigma(1)}<a_{\sigma(2)}<\dots  <a_{\sigma(k)}\rbrace \qquad \sigma \in S_k.
\]
A partition $k=k_1+k_2$ gives rise to a decomposition
\[
F=F_1\times F_2
\]
where $F_1$ is a fixed component of a Nakajima variety $\naka(\vi_1,\w_1)$ with respect of the action of the torus
\[
A_1=\lbrace ( a_{\sigma(1)},a_{\sigma(2)},\dots ,a_{\sigma(k_1)} ) \mid a_i\in \Ci^\times  \rbrace\hookrightarrow A_{\w_1}
\]
and similarly $F_2$ is a fixed component of another variety $\naka(\vi_2,\w_2)$ with respect of the action of
\[
A_2=\lbrace (a_{\sigma(k_1+1)},a_{\sigma(k_1+2)},\dots ,a_{\sigma(k)})  \mid a_i\in \Ci^\times \rbrace\hookrightarrow A_{\w_2}.
\]
By construction, $A=A_1\times A_2$ and we get the following commutative diagram
\[
\begin{tikzcd}
F_1\times F_2 \arrow[d, hookrightarrow]\arrow[r, equal] & F\arrow[d, hookrightarrow]\\
\naka(\vi_1,\w_1)\times \naka(\vi_2,\w_2)\arrow[r, hookrightarrow] & \naka(\vi,\w)
\end{tikzcd}
\]
where all the maps are $A_\w$ equivariant. As usual, the product $\naka(\vi_1,\w_1)\times \naka(\vi_2,\w_2)$ can be also seen as a fixed component of $\naka(\vi,\w)$, this time with respect to the action of a two dimensional torus inside $A$.
Let us now introduce the chambers
\[
\mathfrak{C}^1_\sigma:=\lbrace a_{\sigma(1)}<a_{\sigma(2)}<\dots  <a_{\sigma(k_1)}\rbrace\qquad \mathfrak{C}^2_\sigma:=\lbrace a_{\sigma(k_1+1)}<a_{\sigma(k_1+2)}<\dots  <a_{\sigma(k)}\rbrace
\]
for the actions of $A_1$ and $A_2$ on $\naka(\vi_1,\w_1)$ and $\naka(\vi_2,\w_2)$ respectively. 
\begin{thm}
\label{theorem shuffle stable envelopes}
With the notation above, we have 
\[
\text{Stab}^{\bm\psi}_{\mathfrak{C}_\sigma}(F)=\Shuffle\left( \frac{e(T^*\Rep(\vi,\w)[-1]) \text{Stab}^{\bm\psi}_{\mathfrak{C}^1_\sigma}(F_1) \text{Stab}^{\bm\psi}_{\mathfrak{C}^2_\sigma}(F _2)}{e(\g_{\vi}[-1]) e({\hbar}\g_{\vi}[-1])} \right).
\]
\end{thm}

Notice that the kernel $e(T^*\Rep(\vi,\w)[-1])/(e(\g_{\vi}[-1])e({\hbar}\g_{\vi}[-1])$ of the shuffle formula above coincides with $e(N[-1])$, the Euler class of the negative part of the normal bundle defined in Section \ref{Subsubsection tangent and normal classes}.  
\begin{proof}
Consider the diagram 
\[
\begin{tikzcd}
H_{T_{\w_1}}(F_1)\otimes_\Bbbk H_{T_{\w_2}}(F_2)=H_{T_{\w}}(F)\arrow[rrd, bend left, "\text{Stab}_{\mathfrak{C}_\sigma}(F)"]\arrow[dr, "\text{Stab}_{\mathfrak{C}^1_\sigma}(F_1)\otimes  \text{Stab}_{\mathfrak{C}^2_\sigma}(F _2)"] & & \\
& H_{T_{\w_1}}(\naka(\vi_1,\w_1))\otimes_\Bbbk H_{T_{\w_2}}(\naka(\vi_2,\w_2))\arrow[d, "\bm\psi\otimes \bm\psi"]\arrow[r, "\text{Stab}"] & H_{T_\w}(\naka(\vi,\w))\arrow[d, "\bm\psi"]\\
& H_{T_{\w_1}}(\RG(\vi_1,\w_1))\otimes_\Bbbk H_{T_{\w_2}}(\RG(\vi_2,\w_2)) \arrow[r, "\mathsf{m}_\tau"] & H_{T_\w}(\RG(\vi,\w))
\end{tikzcd}
\]
Here $\text{Stab}$ stands for $\text{Stab}_{\mathfrak{C}_\sigma/\mathfrak{C}^1_\sigma\times \mathfrak{C}^2_\sigma}$.
Commutativity of the central square is exactly the statement of Theorem \ref{main theorem} while commutativity of the upper triangle follows from Lemma \ref{triangle lemma}.
Therefore, applying Theorem \ref{CoHA multiplication formula} we get 
\begin{align*}
    \text{Stab}^{\bm\psi}_{\mathfrak{C}}
    &= \frac{\bm{\psi}\circ \text{Stab}_{\mathfrak{C}}}{e({\hbar}\g_\vi)}\\
    &=\frac{\mathsf{m}_\tau(\bm{\psi}\circ \text{Stab}_{\mathfrak{C}^1_\sigma}(F_1)\otimes  \bm{\psi}\circ \text{Stab}_{\mathfrak{C}^2_\sigma}(F _2))}{e({\hbar}\g_\vi)}\\
    &=\frac{1}{e({\hbar}\g_\vi)}\Shuffle\left( \frac{e(T^*\Rep(\vi,\w)[-1]) e({\hbar}\g_{\vi}[1])(\bm{\psi}\circ \text{Stab}_{\mathfrak{C}^1_\sigma}(F_1)) (\bm{\psi}\circ\text{Stab}_{\mathfrak{C}^2_\sigma}(F _2))}{e(\g_{\vi}[-1]) } \right)\\
    &=\Shuffle\left( \frac{e(T^*\Rep(\vi,\w)[-1])(\bm{\psi}\circ \text{Stab}_{\mathfrak{C}^1_\sigma}(F_1)) (\bm{\psi}\circ\text{Stab}_{\mathfrak{C}^2_\sigma}(F _2))}{e(\g_{\vi}[-1])e({\hbar}\g_{\vi}[-1])e({\hbar}\g_{\vi_1})e({\hbar}\g_{\vi_2}) } \right)\\
    &=\Shuffle\left( \frac{e(T^*\Rep(\vi,\w)[-1]) \text{Stab}^{\bm\psi}_{\mathfrak{C}^1_\sigma}(F_1) \text{Stab}^{\bm\psi}_{\mathfrak{C}^2_\sigma}(F _2)}{e(\g_{\vi}[-1]) e({\hbar}\g_{\vi}[-1])} \right).
\end{align*}
In the penultimate step, we used the fact that $e({\hbar}\g_\vi)$ is symmetric to move it inside the shuffle, and then the factorization $e({\hbar}\g_\vi)=e({\hbar}\g_\vi[1])e({\hbar}\g_\vi[-1])e({\hbar}\g_{\vi_1})e({\hbar}\g_{\vi_2})$ associated to the decomposition $\g_\vi=\g_{\vi_1}\oplus\g_{\vi_2}\oplus\g_\vi[1]\oplus \g_\vi[-1]$ to simplify it with the class $e({\hbar}\g_{\vi}[1])$ in the numerator.
\end{proof}

\newpage 
\printbibliography

\end{document}